\documentclass[12pt]{article} 
\usepackage{amsmath}
\usepackage{amssymb}
\usepackage{url}

\newcommand{\qone}{\frac{m(m-1)}{2} <r \le n }
\newcommand{\qnotone}{\frac{m(m-1)-2qr}{q(q-1)}< n }

\newcommand{\bigfloor}[1]{{\biggl\lfloor {#1} \biggr\rfloor}}
\newcommand{\bigceiling}[1]{{\biggl\lceil {#1} \biggr\rceil}}

\newcommand{\ceil}[1]{\left\lceil {#1}\right\rceil}
\newcommand{\floor}[1]{\left\lfloor{#1}\right\rfloor}

\newcommand{\maxrf}[2]{{\rm maxrf}(#1,#2)}

\newcommand{\tth}{{\rm th }}

\newcommand{\sst}{{\rm st }}

\newcommand{\OBS}{{\rm OBS}}
\newcommand{\COL}{\mathord{\textit{COL}}}

\newcommand{\ang}[1]{\langle#1\rangle}

\newcommand{\kstar}{{\textstyle *}}

\newcommand{\natt}{{\mathbb{N}}}

\newcommand{\xvec}[1]{\ifcase 3{#1} {\ang {x_1,x_2,x_3} } \else 
\ifcase 4{#1} {\ang{x_1,x_2,x_3,x_4}} \else {\ang {x_1,\ldots,x_{#1}}}\fi\fi}
\newcommand{\yvec}[1]{\ifcase 3{#1} {\ang {y_1,y_2,y_3} } \else 
\ifcase 4{#1} {\ang{y_1,y_2,y_3,y_4}} \else {\ang {y_1,\ldots,y_{#1}}}\fi\fi}
\newcommand{\zvec}[1]{\ifcase 3{#1} {\ang {z_1,z_2,z_3} } \else 
\ifcase 4{#1} {\ang{z_1,z_2,z_3,z_4}} \else {\ang {z_1,\ldots,z_{#1}}}\fi\fi}
\newcommand{\vecc}[2]{\ifcase 3{#2} {\ang { {#1}_1,{#1}_2,{#1}_3 } } \else
\ifcase 4{#1} {\ang { {#1}_1,{#1}_2,{#1}_3,{#1}_{4} } }
\else {\ang { {#1}_1,\ldots,{#1}_{#2}}}\fi\fi}
\newcommand{\veccd}[3]{\ifcase 3{#2} {\ang { {#1}_{{#3}1},{#1}_{{#3}2},{#1}_{{#3}3} } } \else
\ifcase 4{#1} {\ang { {#1}_{{#3}1},{#1}_{{#3}2},{#1}_{#3}3},{#1}_{{#3}4} }
\else {\ang { {#1}_{{#3}1},\ldots,{#1}_{{#3}{#2}}}}\fi\fi}
\newcommand{\veccz}[2]{\ifcase 3{#2} {\ang { {#1}_0,{#1}_2,{#1}_3 } } \else
\ifcase 4{#1} {\ang { {#1}_0,{#1}_2,{#1}_3,{#1}_{4} } }
\else {\ang { {#1}_0,\ldots,{#1}_{#2}}}\fi\fi}
\newcommand{\xve}[1]{\ifcase 3{#1} {x_1,x_2,x_3} \else 
\ifcase 4{#1} {x_1,x_2,x_3,x_4} \else {x_1,\ldots,x_{#1}}\fi\fi}
\newcommand{\yve}[1]{\ifcase 3{#1} {y_1,y_2,y_3} \else 
\ifcase 4{#1} {y_1,y_2,y_3,y_4} \else {y_1,\ldots,y_{#1}}\fi\fi}
\newcommand{\zve}[1]{\ifcase 3{#1} {z_1,z_2,z_3} \else 
\ifcase 4{#1} {z_1,z_2,z_3,z_4} \else {z_1,\ldots,z_{#1}}\fi\fi}
\newcommand{\ve}[2]{\ifcase 3#2 {{#1}_1,{#1}_2,{#1}_3} \else
\ifcase 4#2 {{#1}_1,{#1}_2,{#1}_3,{#1}_{4}}
\else {{#1}_1,\ldots,{#1}_{#2}}\fi\fi}
\newcommand{\ved}[3]{\ifcase 3#2 {{#1}_{{#3}1},{#1}_{{#3}2},{#1}_{{#3}3}} \else
\ifcase 4#2 {{#1}_{{#3}1},{#1}_{{#3}2},{#1}_{{#3}3},{#1}_{{#3}4}}
\else {{#1}_{{#3}1},\ldots,{#1}_{{#3}{#2}}}\fi\fi}
\newcommand{\fuve}[3]{
\ifcase 3#2
{{#3}({#1}_1),{#3}({#1}_2,{#3}({#1}_3)} \else
\ifcase 4#2
{{#3}({#1}_1),{#3}({#1}_2),{#3}({#1}_3),{#3}({#1}_4)}
\else
{{#3}({#1}_1),\ldots,{#3}({#1}_{#2})}\fi\fi}
\newcommand{\setmathchar}[1]{\ifmmode#1\else$#1$\fi}
\newcommand{\vlist}[2]{%
	\setmathchar{%
		\compound#2\one{#2}\two
		\ifcompound
			({#1}_1,\ldots,{#1}_{#2})
		\else
			\ifcat N#2
				({#1}_1,\ldots,{#1}_{#2})
			\else
				\ifcase#2
					({#1}_0)\or
					({#1}_1)\or
					({#1}_1,{#1}_2)\or 
					({#1}_1,{#1}_2,{#1}_3)\or
					({#1}_1,{#1}_2,{#1}_3,{#1}_4)\else 
					({#1}_1,\ldots,{#1}_{#2})
				\fi
			\fi
		\fi}}

\newif\ifcompound
\def\compound#1\one#2\two{%
	\def\one{#1}
	\def\two{#2}
	\if\one\two
		\compoundfalse
	\else
		\compoundtrue
	\fi}

\newcommand{\xwe}[1]{\ifcase 3{#1} {x_1\wedge x_2\wedge x_3} \else 
\ifcase 4{#1} {x_1\wedge x_2\wedge x_3\wedge x_4} \else {x_1\wedge \cdots \wedge
x_{#1}}\fi\fi}
\newcommand{\we}[2]{\ifcase 3#2 {\ang { {#1}_1\wedge {#1}_2\wedge {#1}_3 } } \else
\ifcase 4{#1} {\ang { {#1}_1\wedge {#1}_2\wedge {#1}_3\wedge {#1}_{4} } }
\else {\ang { {#1}_1\wedge \cdots\wedge {#1}_{#2}}}\fi\fi}

\newcommand{\st}{\mathrel{:}}

\newcommand{\into}{\rightarrow}

\newcommand{\eps}{\varepsilon}
\newcommand{\es}{\emptyset}

\newcommand{\union}{\cup}

\newcommand{\s}[1]{\s_{#1}}

\newcommand{\monus}{\;\raise.5ex\hbox{{${\buildrel
    \ldotp\over{\hbox to 6pt{\hrulefill}}}$}}\;}

\newcounter{savenumi}

\newtheorem{theoremfoo}{Theorem}[section] %by chapter in report style
\newenvironment{theorem}{\pagebreak[1]\begin{theoremfoo}}{\end{theoremfoo}}

\newtheorem{lemmafoo}[theoremfoo]{Lemma}
\newenvironment{lemma}{\pagebreak[1]\begin{lemmafoo}}{\end{lemmafoo}}
\newtheorem{conjecturefoo}[theoremfoo]{Conjecture}

\newtheorem{conventionfoo}[theoremfoo]{Convention}

\newtheorem{porismfoo}[theoremfoo]{Porism}

\newtheorem{gamefoo}[theoremfoo]{Game}

\newtheorem{corollaryfoo}[theoremfoo]{Corollary}
\newenvironment{corollary}{\pagebreak[1]\begin{corollaryfoo}}{\end{corollaryfoo}}

\newtheorem{claimfoo}[theoremfoo]{Claim}

\newtheorem{openfoo}[theoremfoo]{Open Problem}

\newtheorem{exercisefoo}{Exercise}

\newcommand{\fig}[1] %usage:\fig{file}
{
 \begin{figure}
 \begin{center}
 \input{#1}
 \end{center}
 \end{figure}
}

\newtheorem{potanafoo}[theoremfoo]{Potential Analogue}

\newtheorem{notefoo}[theoremfoo]{Note}
\newenvironment{note}{\pagebreak[1]\begin{notefoo}\rm}{\end{notefoo}}

\newtheorem{notabenefoo}[theoremfoo]{Nota Bene}

\newtheorem{nttn}[theoremfoo]{Notation}
\newenvironment{notation}{\pagebreak[1]\begin{nttn}\rm}{\end{nttn}}

\newtheorem{empttn}[theoremfoo]{Empirical Note}

\newtheorem{examfoo}[theoremfoo]{Example}
\newenvironment{example}{\pagebreak[1]\begin{examfoo}\rm}{\end{examfoo}}

\newtheorem{dfntn}[theoremfoo]{Def}
\newenvironment{definition}{\pagebreak[1]\begin{dfntn}\rm}{\end{dfntn}}

\newtheorem{propositionfoo}[theoremfoo]{Proposition}

\newenvironment{proof}
    {\pagebreak[1]{\narrower\noindent {\bf Proof:\quad\nopagebreak}}}{\QED}

\newcommand{\yyskip}{\penalty-50\vskip 5pt plus 3pt minus 2pt}
\newcommand{\blackslug}{\hbox{\hskip 1pt
        \vrule width 4pt height 8pt depth 1.5pt\hskip 1pt}}
\newcommand{\QED}{{\penalty10000\parindent 0pt\penalty10000
        \hskip 8 pt\nolinebreak\blackslug\hfill\lower 8.5pt\null}
        \par\yyskip\pagebreak[1]}

\newcommand{\BBB}{{\penalty10000\parindent 0pt\penalty10000
        \hskip 8 pt\nolinebreak\hbox{\ }\hfill\lower 8.5pt\null}
        \par\yyskip\pagebreak[1]}

\newtheorem{factfoo}[theoremfoo]{Fact}

\newenvironment{block}{\begin{list}{\hbox{}}{\leftmargin 1em
    \itemindent -1em \topsep 0pt \itemsep 0pt \partopsep 0pt}}{\end{list}}

\newcommand{\Gnm}{G_{n,m}}
\newcommand{\Gmn}{G_{m,n}}
\usepackage{fullpage}

\begin{document}       
\title{Rectangle Free Coloring of Grids}

\author{
{Stephen Fenner}
\thanks{University of South Carolina,
Department of Computer Science and Engineering,
Columbia, SC, 29208
\texttt{fenner@cse.sc.edu},
Partially supported by NSF CCF-0515269
}
\\ {\small Univ of South Carolina}
\and
{William Gasarch}
\thanks{University of Maryland,
Dept. of Computer Science,
        College Park, MD\ \ 20742.
\texttt{gasarch@cs.umd.edu}
}
\\ {\small Univ. of MD at College Park}
\and
{Charles Glover}
\thanks{Booz Allen Hamilton,
134 National Business Parkway
        Annaopolis Junction, MD\ \ 20701.
\texttt{glover\_charles@bah.com}
}
\\ {\small Univ. of MD at College Park}
\and
{Semmy Purewal}
\thanks{University of North Carolina at Ashville,
Department of Computer Science,
Ashville, NC 28804
\texttt{tspurewal@gmail.com}
}
\\ {\small Univ of NC at Ashville}
}

\date{}

\maketitle

\begin{abstract}
Let $\Gnm$ be the grid $[n]\times[m]$.
$\Gnm$ is \emph{$c$-colorable} if there is a function
$\chi: \Gnm \rightarrow [c]$  such that
there are no rectangles with all four corners
the same color.  
We ask {\it for which values of $n,m,c$ is $\Gnm$ $c$-colorable?}
We determine 
(1) \emph{exactly} which grids are 2-colorable,
(2) \emph{exactly} which grids are 3-colorable,
(2) \emph{exactly} which grids are 4-colorable.
Our main tools are combinatorics and finite fields.

Our problem has two motivations: (1) (ours) A Corollary of the Gallai-Witt theorem states
that, for all $c$, there exists $W=W(c)$  such that any $c$-coloring of $[W]\times[W]$
has a monochromatic square. The bounds on $W(c)$ are enormous. Our relaxation of the problem
to rectangles yields much smaller bounds.
(2) Colorings grids to avoid a rectangle is equivalent to coloring the edges
of a bipartite graph to avoid a monochromatic $K_{2,2,}$. Hence our work is related
to bipartite Ramsey Numbers. 
\end{abstract}

\vfill\eject

\tableofcontents

\section{Introduction}

\begin{notation}
If $n\in\natt$ then $[n]=\{1,\ldots,n\}$.
If $n,m\in\natt$ then $\Gnm$ is the grid $[n]\times[m]$.
\end{notation}

The Gallai-Witt theorem\footnote{It was attributed to Gallai in
  \cite{radogerman} and \cite{radoenglish}; Witt proved the theorem in
\cite{witt}.} (also called the multi-dimensional 
Van Der Waerden theorem) has the following corollary:
{\it For all $c$, there exists $W=W(c)$ such that, for all $c$-colorings
of $[W]\times [W]$ there exists a monochromatic square.}
The classical proof of the theorem gives very large upper bounds
on $W(c)$. Despite some improvements~\cite{gridiowa},
the known bounds on $W(c)$ are still quite large.
If we relax the problem to seeking a {\it monochromatic rectangle}
then we can obtain far smaller bounds. 

\begin{definition}
A \emph{rectangle} of $\Gnm$ is a subset
of the form $\{(a,b),(a+c_1,b),(a+c_1,b+c_2),(a,b+c_2)\}$ for some $a, b, c_1, c_2 \in \natt$.  
A grid $\Gnm$ is \emph{$c$-colorable} if there is a function
$\chi: \Gnm \rightarrow [c]$  such that
there are no rectangles with all four corners
the same color.
\end{definition}

Not all grids have
$c$-colorings.  As an example, for any $c$ clearly $G_{c+1,c^{c+1}+1}$
does not have a $c$-coloring by two applications of the pigeonhole
principle.  In this paper, we ask the following. Fix $c$.

\bigskip

\centerline{\it For which values of $n$ and $m$ is $\Gnm$ $c$-colorable?}

\bigskip

\begin{definition}
Let $n,m,n',m'\in \natt$.
$\Gmn$ {\it contains } $G_{n',m'}$ if $n'\le n$ and $m'\le m$.
$\Gmn$ {\it is contained in } $G_{n',m'}$ if $n\le n'$ and $m\le m'$.
Proper containment means that at least one of the inequalities is strict.
\end{definition}

Clearly, if $\Gnm$ is $c$-colorable, then all grids that it contains are $c$-colorable.
Likewise, if $\Gnm$ is
not $c$-colorable then all grids that contain it are not $c$-colorable.  

\begin{definition}
Fix $c \in \natt$. $\OBS_c$ is the set of all grids $\Gnm$ such that
$\Gnm$ is not $c$-colorable but all grids properly contained in $\Gmn$ are 
$c$-colorable.
$\OBS_c$ stands for {\it Obstruction Sets}.
\end{definition}

We leave the proof of the following theorem to the reader.

\begin{theorem}\label{th:obs}
Fix $c \in \natt$. A grid $\Gnm$ is $c$-colorable iff
it does not contain any element of $\OBS_c$.
\end{theorem}

By Theorem~\ref{th:obs} we can rephrase the question of finding
which grids are $c$-colorable:

\bigskip

\centerline{\it What is $\OBS_c$?}

\bigskip

Note that if $\Gnm \in \OBS_c$, then $\Gmn \in \OBS_c$.

Our problem has another motivation involving the Bipartite Ramsey Theorem
which we now state.
\begin{theorem}
For all $L$, for all $c$,  there exists $n$ such that for any $c$-coloring of
the edges of $K_{n,n}$ there exists a monochromatic $K_{L,L}$.
\end{theorem}

We now state a corollary of the Bipartite Ramsey theorem and a statement about
grid colorings that is easily seen to be equivalent to it.

\begin{enumerate}
\item
For all $c$ there exists $n$ such that for any $c$-coloring of the edges of $K_{n,n}$
there exists a monochromatic $K_{2,2}$.
\item
For all $c$ there exists $n$ such that for any $c$-coloring of $G_{n,n}$ 
there exists a monochromatic rectangle.
\end{enumerate}

One can ask, given $c$, what is $n$?
Beineke and Schwenk \cite{beineke:1975} studied a closely
related problem: What is the minimum value of $n$ such that any
2-coloring of $K_{n,n}$ results in a monochromatic $K_{a,b}$?  In
their work, this minimal value is denoted $R(a,b)$.  Later, Hattingh
and Henning \cite{hattingh:1998} defined $n(a,b)$ as the minimum $n$
for which any 2-coloring of $K_{n,n}$ contains a monochromatic
$K_{a,a}$ or a monochromatic $K_{b,b}$.  

Our results are about $G_{n,m}$ not just $G_{n,n}$ hence they are not
quite the same as the bipartite Ramsey numbers. Even so, we do obtain
some new Bipartite Ramsey Numbers. They are in Section~\ref{se:bipartite}.

The remainder of this paper is organized as follows.  In Section~\ref{se:lower}
we develop tools to show grids {\it are not } $c$-colorable.
In Section~\ref{se:findprop} we develop tools to show grids {\it are} $c$-colorable.
In Section~\ref{se:bounds} we obtain upper and lower bounds on $|\OBS_c|$.
In Section~\ref{se:col2}, ~\ref{se:col3}, and ~\ref{se:col4}  we find $\OBS_2$, $\OBS_3$, and $\OBS_4$  respectively.
In Section~\ref{se:bipartite} we apply the results to find some
new bipartite Ramsey numbers.
We conclude with some open questions.
The appendix contains some sizes of maximum rectangle free sets
(to be defined later).

In a related paper, Cooper, Fenner, and Purewal \cite{cooper:2008}
generalize our problem to multiple dimensions and obtain upper and lower 
bounds on the sizes of the obstruction sets.
In another related paper Molina, Oza, and Puttagunta~\cite{MOP}
have looked at some variants of our questions.

\section{Tools to Show Grids are Not $c$-colorable}\label{se:lower}

\subsection{Using Rectangle Free Sets}

A \emph{rectangle-free subset} $A \subseteq \Gnm$ is a subset that
does not contain a rectangle.  A problem that is
closely related to grid-colorability is that of finding a
rectangle-free subset of maximum cardinality.  This relationship is
illustrated by the following lemma.

\begin{theorem}
\label{th:connection}
If $\Gnm$ is $c$-colorable, then it contains a rectangle-free
subset of size $\lceil \frac{nm}{c} \rceil$.
\end{theorem}

\begin{proof}
A $c$-coloring partitions the elements of $\Gnm$ into $c$
rectangle-free subsets.  By the pigeon-hole principle, one of these
sets must be of size at least $\lceil \frac{nm}{c} \rceil$.
\end{proof}

\begin{definition}
Let $n,m\in\natt$.
$\maxrf n m $ is the size of the maximum
rectangle-free $A\subseteq \Gnm$.
\end{definition}

Finding the maximum cardinality of a rectangle-free subset is
equivalent to a special case of a well-known problem 
of  Zarankiewicz~\cite{zaran} (see \cite{GRS} or \cite{roman} for more information).
The Zarankiewicz function, denoted
$Z_{r,s}(n,m)$, counts the minimum number of edges in a bipartite
graph with vertex sets of size $n$ and $m$ that guarantees a subgraph
isomorphic to $K_{r,s}$.  
Zarankiewicz's problem was to determine $Z_{r,s}(n,m)$.

If $r=s$, the function is denoted
$Z_r(n,m)$.  If one views a grid as an incidence matrix for a
bipartite graph with vertex sets of cardinality $n$ and $m$, then a
rectangle is equivalent to a subgraph isomorphic to $K_{2,2}$.
Therefore the maximum cardinality of a rectangle-free set in $\Gnm$
is $Z_2(n,m) - 1$.  We will use this lemma in its contrapositive form,
i.e., we will often show that $\Gnm$ is not $c$-colorable by showing
that $Z_2(n,m) \leq \lceil \frac{nm}{c} \rceil$.

Reiman~\cite{reiman} proved the following lemma. Roman~\cite{roman} later generalized it.

\begin{lemma}
\label{reimanlemma}
Let $m \leq n \leq \binom{m}{2}$. 
Then $Z_2(n,m) \leq \floor{\frac{n}{2}\left(1+\sqrt{1+4m(m-1)/n}\right)} + 1$.
\end{lemma}

\begin{corollary}
\label{uncolor1}
Let $m \leq n \leq \binom{m}{2}$.
Let $z_{n,m} = \floor{\frac{n}{2}\left(1+\sqrt{1+4m(m-1)/n}\right)}+1$ be the upper-bound on $Z_2(n,m)$ in Lemma \ref{reimanlemma}.
If 
$z_{n,m} \leq \lceil \frac{nm}{c} \rceil$
then $\Gnm$ is not $c$-colorable.
\end{corollary}

Corollary \ref{uncolor1}, and some  2-colorings of grids,
are sufficient to find $\OBS_2$.  
To find $\OBS_3$ and $\OBS_4$, we need more powerful tools to show grids
are not colorable (along with some 3-colorings and 4-colorings of grids).

\begin{definition}
Let $n,m,x_1,\ldots,x_m\in\natt$.
$(x_1,\ldots,x_m)$ is \emph{$(n,m)$-placeable}
if there exists a rectangle-free $A\subseteq \Gnm$
such that, for $1\le j\le m$, there are $x_j$ elements
of $A$ in the $j^\tth$ column.
\end{definition}

\begin{lemma}\label{le:binom}
Let $n,m,x_1,\ldots,x_m\in\natt$ be such
that
$(x_1,\ldots,x_m)$ is $(n,m)$-placeable.
Then
$\sum_{i=1}^m \binom{x_i}{2} \le \binom{n}{2}.$
\end{lemma}

\begin{proof}
Let $A\subseteq \Gnm$ be a set that shows that 
$(x_1,\ldots,x_m)$ is $(n,m)$-placeable.
Let $\binom{A}{2}$ be the set of pairs of elements of $A$.
Let $2^{\binom{A}{2}}$ be the powerset of $\binom{A}{2}$.

Define the function $f:[m]\into 2^{\binom{A}{2}}$ as follows.
For $1\le j\le m$,

$$f(j) = \{ \{a,b\} \st (a,j), (b,j) \in A \}.$$

If $\sum_{j=1}^m |f(j)| > \binom{n}{2}$ then there exists
$j_1\ne j_2$ such that $f(j_1)\cap f(j_2)\ne \es$.
Let $\{a,b\} \in f(j_1)\cap f(j_2)$.
Then
$$\{ (a,j_1), (a,j_2), (b,j_1), (b,j_2) \} \subseteq A.$$
Hence $A$ contains a rectangle.
Since this cannot happen, 
$\sum_{j=1}^m |f(j)| \le \binom{n}{2}$.
Note that $|f(j)|=\binom{x_j}{2}$. Hence
$\sum_{i=1}^m \binom{x_i}{2} \le \binom{n}{2}.$
\end{proof}

\begin{theorem}\label{th:density}
Let $a,n,m \in \natt$. Let $q,r \in \natt$ be such that
$a=qn+r$ with $0\le r\le n$. 
Assume that there exists $A\subseteq \Gnm$ such that $|A|=a$ and $A$ is
rectangle-free.
\begin{enumerate}
\item
If $q\ge 2$ then
$$
n\le \bigfloor{\frac{m(m-1)-2rq}{q(q-1)}}.
$$
\item
If $q=1$ then 
$$
r\le \frac{m(m-1)}{2}.
$$
\end{enumerate}
\end{theorem}

\begin{proof}
The proof for the $q\ge 2$ and the $q=1$ case
begins the same; hence we will not split into cases yet.

Assume that, for $1\le j\le m$, 
 the number of elements of $A$ in the
$j^\tth$ column is $x_j$. Note that  $\sum_{j=1}^m x_j = a$.
By Lemma~\ref{le:binom}
$\sum_{j=1}^m \binom{x_j}{2} \le  \binom{n}{2}$.
We look at the least value that $\sum_{j=1}^n \binom{x_j}{2}$ can have.

Consider the following question:

\noindent
Minimize $\sum_{j=1}^n \binom{x_j}{2}$ 

\medskip

\noindent
Constraints:
\begin{itemize}
\item
$\sum_{j=1}^n x_j = a$.
\item
$x_1,\ldots,x_n$ are natural numbers.
\end{itemize}

One can easily show that this is minimized
when, for all $1\le j\le n$,
$$x_j\in \{ \floor{a/n}, \ceil{a/n}\} = \{q,q+1\}.$$
In order for $\sum_{j=1}^n x_j = a$ we need to
have $n-r$ many $q$'s and $r$ many $q+1$'s.
Hence we obtain

$\sum_{j=1}^n \binom{x_j}{2}$ is at least
$$(n-r)\binom{q}{2} + r\binom{q+1}{2}.$$

Hence we have
$$(n-r)\binom{q}{2} + r\binom{q+1}{2} \le \sum_{j=1}^n \binom{x_j}{2}
\le \binom{m}{2}$$
$$nq(q-1) -rq(q-1) + r(q+1)q \le m(m-1)$$
$$nq(q-1) - rq^2+ rq + rq^2 +rq \le m(m-1)$$
$$nq(q-1) +2rq \le m(m-1)$$

\noindent
{\bf Case 1:} $q\ge 2$. 

Subtract $2rq$ from both sides to obtain
$$nq(q-1) \le m(m-1)-2rq.$$
Since $q-1\ne 0$ we can divide
by $q(q-1)$ to obtain
$$
n\le \bigfloor{\frac{m(m-1)-2rq}{q(q-1)}}.
$$

\bigskip

\noindent
{\bf Case 2:} $q=1$. 

Since $q-1=0$ we get

$$2r \le m(m-1)$$

$$r \le \frac{m(m-1)}{2}.$$
\end{proof}

\begin{corollary}
\label{uncolor2}
Let $m,n \in \natt$.  If there exists an $r$ where
$\frac{m(m-1)}{2} < r \le n$ and $\ceil{\frac{mn}{c}} = n+r$, then
$G_{m,n}$ is not $c$-colorable.
\end{corollary}

\begin{corollary}
\label{uncolor3}
Let $n,m\in\natt$.  Let $\lceil \frac{nm}{c} \rceil = qn+r$ for some $0
\le r \le n$ and $q \ge  2$.  If $\frac{m(m-1)-2qr}{q(q-1)} < n$ then
$\Gnm$ is not $c$-colorable.
\end{corollary}

We now show that, for all $c$,  $G_{c^2,c^2+c+1}$  is not $c$-colorable.
This is particularly interesting because 
by Theorem~\ref{th:primepowersqplus}, for $c$ a prime power,
$G_{c^2,c^2+c}$ is $c$-colorable.

\begin{corollary}\label{co:lowerboundcsq}
For all $c\ge 2$ $G_{c^2,c^2+c+1}$ is not $c$-colorable.
\end{corollary}

\begin{proof}
If $G_{c^2,c^2+c+1}$ is $c$-colorable then there exists a rectangle free subset
of $G_{c^2,c^2+c+1}$ of size $\frac{c^2(c^2+c+1)}{c} = c(c^2+c+1)$.
Let $a=c(c^2+c+1)$, $n=c^2+c+1$, and $m=c^2$ in Theorem~\ref{th:density}.
Then $q=c$ and $r=0$. 
By that lemma we have

$$n\le \bigfloor{\frac{m(m-1)-2rq}{q(q-1)}}$$

we should have

$$c^2+c+1 \le \bigfloor{\frac{c^2(c^2-1)}{c(c-1)}}= c(c+1)=c^2+c$$

This is a contradiction.
\end{proof}

\begin{note}
In the Appendix we use the results of this section to find
the sizes of maximum rectangle free sets.
\end{note}

\begin{corollary}\label{uncolor2a}~
\begin{enumerate}
\item
Let $c\ge 2$ and $1 \le c' < c$.  
Let $n> \frac{c}{c'}\binom{c+c'}{2}$.
Then $G_{n,c+c'}$ is not $c$-colorable.
\item
Let $c\ge 2$ and $1 \le c' < c$.  
Let $m> \frac{c}{c'}\binom{c+c'}{2}$.
Then $G_{c+c',m}$ is not $c$-colorable.
(This follows immediately from part $a$.)
\end{enumerate}
\end{corollary}

\begin{proof}
Assume, by way of contradiction, that $G_{n,c+c'}$ is $c$-colorable.
Then there is a rectangle free set of size

$$\ceil{\frac{n(c+c')}{c}} = \ceil{n + \frac{c'n}{c}} = n+ \ceil{\frac{c'}{c}n}.$$

Since $c'<c$ we have

$$\ceil{\frac{n(c+c')}{c}} = n+ \ceil{\frac{c'}{c}n}\le n+ \ceil{\frac{c-1}{c}n} = n+\ceil{n-\frac{n}{c}}.$$
The premise of this corollary implies $c<n$. Hence 
$$\ceil{\frac{n(c+c')}{c}} \le  n+\ceil{n-\frac{n}{c}}\le 2n-1.$$
Therefore when we divide $n$ into 
$r=\ceil{\frac{c'n}{c}}$.

$$\ceil{\frac{n(c+c')}{c}}= n + \ceil{\frac{c'n}{c}}.$$

We want to apply Corollary~\ref{uncolor2} with $m=c+c'$ and $r=\ceil{\frac{c'n}{c}}$.
We need 

$$\frac{m(m-1)}{2} < r \le n.$$

$$\frac{(c+c')(c+c'-1)}{2} < \ceil{\frac{c'n}{c}} \le n.$$

The second inequality is obvious. 
The first inequality follows from $n> \frac{c}{c'}\binom{c+c'}{2}$.

\end{proof}

\begin{note}
In the Appendix we use the results of this section to find
the sizes of maximum rectangle free sets.
\end{note}

\subsection{Using maxrf}

\begin{notation}
If $n,m\in \natt$ and $A\subseteq\Gnm$.
\begin{enumerate}
\item
We will denote that $(a,b)\in A$
by putting an $R$ in the $(a,b)$ position.
\item
For $1\le j\le m$, 
$x_j$ is the number of elements of $A$ in column $j$.
\item
For $1\le j\le m$, $C_j$ is the set of rows $r$ such that
$A$ has an element in the $r^\tth$ row of column $j$.
Formally
$$C_j = \{ r \st (r,j)\in A \}.$$
\end{enumerate}
\end{notation}

\begin{definition}
Let $n,m\in\natt$ and $A\subseteq \Gnm$.
Let $1\le i_1 < i_2 \le n$.
$C_{i_1}$ and $C_{i_2}$ \emph{intersect} 
if $C_{i_1}\cap C_{i_2} \ne \es$.
\end{definition}

\begin{lemma}\label{le:maxrf1}
Let $n,m\in\natt$. 
Let $x_1\le n$.
Assume $(x_1,\ldots,x_m)$ is $(n,m)$-placeable via $A$.
Then
$$|A| \le x_1 + m-1 +  \maxrf {n-x_1}{m-1}.$$
\end{lemma}

\begin{proof}
The picture in Table~\ref{ta:3grid} portrays what might happen.
We use double lines to partition the grid
in a way that will be helpful later.

\begin{table}[htbp]
\[
\begin{array}{|c||c||c|c|c|c|c|c|c|c|}
\hline 
  & 1 & 2 & 3&4&5& \ldots & j & \cdots  & m\cr
\hline
\hline
1 & R & R &  & & & \cdots &   & \cdots        &\cr
\hline
2 & R &   & R& &   &\cdots &  &\cdots &\cr
\hline
3 & R &   &  &R&   &\cdots & &\cdots &\cr
\hline
\vdots  & R &   &  & & R &\ldots & &\cdots &\cr
\hline
x_1 & R & ? & ?& & ? &\cdots &?&\cdots & ?\cr
\hline
\hline
x_1+1 &   & ? &? &?& ? &\cdots &?&\cdots  &?\cr
\hline
x_1+2 &   & ? &?&?&?  &\cdots &?&\cdots  &?\cr
\hline
\vdots &   &?  &?&?&?  &\cdots &?&\cdots  &?\cr
\hline
n &   & ? &?&?& ? &\cdots &?&\cdots  &?\cr
\hline
\end{array}
\]
\caption{The Grid in Three Parts\label{ta:3grid}}
\end{table}

\noindent
{\bf Part 1:}
The first column. This has $x_1$ elements of $A$ in it.

\medskip

\noindent
{\bf Part 2:}
Consider the grid consisting of rows $1,\ldots,x_1$
and columns $2,\ldots,m$.
Look at the $j^\tth$ column, $2\le j\le m$ in this grid.
For each such $j$, this column has at most one element in $A$
(else there would be a rectangle using the first column).
Hence the total number of elements of $A$ from this part of
the grid is $m-1$.
(We drew them in a diagonal pattern though this is not required.)

\medskip

\noindent
{\bf Part 3:}
The bottom most $n-x_1$ elements of the right most $m-1$ columns.
This clearly has $\le \maxrf {n-x_1} {m-1}$ elements in it.
We do not know which elements will be taken so we just use ?'s.

Taking all the parts into account we obtain

$$|A|\le x_1+ (m-1) + \maxrf {n-x_1} {m-1}.$$
\end{proof}

\section{Tools for Finding $c$-colorings}\label{se:findprop}

\subsection{Strong $c$-colorings and Strong $(c,c')$-colorings}

\begin{definition}
Let $c,c',n,m\in\natt$ and
let $\chi:\Gnm \into [c]$.
Assume $c'\le c$.
\begin{enumerate}
\item
A {\it half-mono rectangle with respect to $\chi$} is a rectangle
where the left corners are the same color and the right corners
are the same color.
\item
$\chi$ is a \emph{strong $c$-coloring} if there
are no half-mono rectangles.
\item
$\chi$ is a \emph{strong $(c,c')$-coloring}
if for any half-mono rectangle
the color of the left corners and the right corners are (1) different, and
(2) in $[c']$.
\end{enumerate}
\end{definition}

\begin{example}\label{ex:strong}~
\begin{enumerate}
\item
Table~\ref{ta:s58} is a strong $4$-coloring of $G_{5,8}$.

\begin{table}[htbp]
\[
\begin{array}{|c|c|c|c|c|c|c|c|}
\hline
1 & 1 & 1 & 4 & 1 & 1 & 4 & 4\cr
\hline
2 & 2 & 4 & 1 & 2 & 4 & 1 & 4\cr
\hline
3 & 4 & 2 & 2 & 4 & 2 & 4 & 1\cr
\hline
4 & 3 & 3 & 3 & 4 & 4 & 2 & 2\cr
\hline
4 & 4 & 4 & 4 & 3 & 3 & 3 & 3\cr
\hline
\end{array}
\]
\caption{Strong 4-coloring of $G_{5,8}$\label{ta:s58}}
\end{table}

\item
Table~\ref{ta:s46} is a strong $3$-coloring of $G_{4,6}$.

\begin{table}[htbp]
\[
\begin{array}{|c|c|c|c|c|c|}
\hline
1 & 1 & 3 & 1 & 3 & 3 \cr
\hline
2 & 3 & 1 & 3 & 1 & 3 \cr
\hline
3 & 2 & 2 & 3 & 3 & 1 \cr
\hline
3 & 3 & 3 & 2 & 2 & 2 \cr
\hline
\end{array}
\]
\caption{Strong 3-coloring of $G_{4,6}$\label{ta:s46}}
\end{table}

\item
Table~\ref{ta:s42615} is a strong $(4,2)$-coloring of $G_{6,15}$.

\begin{table}[htbp]
\[
\begin{array}{|c|c|c|c|c|c|c|c|c|c|c|c|c|c|c|}
\hline
1 & 1 & 1 & 1 & 1 & 3 & 3 & 3 &2&3 &3 &2 &2  & 2 & 2\cr
\hline
1 & 2 & 2 & 2 & 2 & 1 & 1 & 1 &1&4 &4 &3 &3 &3 & 2 \cr
\hline
2 & 1 & 3 & 3 & 2 & 1 & 2 & 2 &2&1 &1 &1 &4 &4 & 3 \cr
\hline
2 & 2 & 1 & 4 & 3 & 2 & 1 & 4 &3&1 &2 &2 &1  &1 & 4 \cr
\hline
3 & 3 & 2 & 1 & 4 & 2 & 2 & 1 &4&2 &1 &4 &1  &2 &1 \cr
\hline
4 & 4 & 4 & 2 & 1 & 4 & 4 & 2 &1&2 &2 &1 &2 & 1 &1 \cr
\hline
\end{array}
\]
\caption{Strong $(4,2)$-coloring of $G_{6,15}$\label{ta:s42615}}
\end{table}

\item
Table~\ref{ta:s6286} is a strong $(6,2)$-coloring of $G_{8,6}$.

\begin{table}[htbp]
\[
\begin{array}{|c|c|c|c|c|c|}
\hline
1 & 1 & 2 & 2 & 3 & 6 \cr
\hline
1 & 2 & 1 & 2 & 4 & 5 \cr
\hline
2 & 1 & 2 & 1 & 5 & 4\cr
\hline
2 & 2 & 1 & 1 & 6 & 3\cr
\hline
3 & 4 & 5 & 6 & 1 & 2\cr
\hline
4 & 5 & 6 & 4 & 1 & 1\cr
\hline
5 & 6 & 3 & 3 & 1 & 2\cr
\hline
6 & 3 & 4 & 5 & 1 & 2\cr
\hline
\end{array}
\]
\caption{Strong $(6,2)$-coloring of $G_{8,6}$\label{ta:s6286}}
\end{table}

\item
Table~\ref{ta:s53828} is a $(5,3)$-coloring of $G_{8,28}$.

\begin{table}[htbp]
\[
\begin{array}{|c|c|c|c|c|c|c|c|c|c|c|c|c|c|c|c|c|c|c|c|c|c|c|c|c|c|c|c|}
\hline
1 & 1 & 1 & 1 & 1 & 1 & 1  & 5 & 5 & 5 & 5 & 3 & 2 & 4 & 3 & 4 & 3 & 2 & 3 & 4 & 3 & 2 & 3 & 3 & 2 & 2 & 2 & 2 \cr
\hline
1 & 2 & 2 & 2 & 2 & 2 & 2  & 1 & 1 & 1 & 1 & 1 & 1 & 5 & 4 & 5 & 4 & 3 & 4 & 3 & 4 & 3 & 3 & 4 & 3 & 3 & 3 & 2 \cr
\hline
2 & 1 & 3 & 3 & 3 & 3 & 2  & 1 & 2 & 2 & 2 & 2 & 2 & 1 & 1 & 1 & 1 & 1 & 5 & 5 & 5 & 4 & 4 & 3 & 4 & 3 & 4 & 3 \cr
\hline
2 & 2 & 1 & 4 & 4 & 4 & 3  & 2 & 1 & 3 & 3 & 3 & 3 & 1 & 2 & 2 & 2 & 2 & 1 & 1 & 1 & 1 & 5 & 5 & 5 & 4 & 3 & 3 \cr
\hline
3 & 3 & 2 & 1 & 5 & 3 & 3  & 2 & 2 & 1 & 4 & 4 & 4 & 2 & 1 & 3 & 3 & 3 & 1 & 2 & 2 & 2 & 1 & 1 & 1 & 5 & 5 & 4 \cr
\hline
3 & 4 & 3 & 2 & 1 & 5 & 4  & 3 & 3 & 2 & 1 & 5 & 3 & 2 & 2 & 1 & 5 & 4 & 2 & 1 & 3 & 3 & 1 & 2 & 2 & 1 & 1 & 5 \cr
\hline
4 & 3 & 4 & 3 & 2 & 1 & 5  & 3 & 4 & 3 & 2 & 1 & 5 & 3 & 3 & 2 & 1 & 5 & 2 & 2 & 1 & 5 & 2 & 1 & 3 & 1 & 2 & 1 \cr
\hline
5 & 5 & 5 & 5 & 3 & 2 & 1  & 4 & 3 & 4 & 3 & 2 & 1 & 3 & 5 & 3 & 2 & 1 & 3 & 3 & 2 & 1 & 2 & 2 & 1 & 2 & 1 & 1 \cr
\hline
\end{array}
\]
\caption{Strong $(5,3)$-coloring of $G_{8,28}$\label{ta:s53828}}
\end{table}

\end{enumerate}
\end{example}

\begin{theorem}\label{th:strong}
Let $c,c',n,m\in\natt$. Let $x=\floor{c/c'}$.
If $\Gnm$ is strongly $(c,c')$-colorable
then $G_{n,xm}$ is $c$-colorable.
\end{theorem}

\begin{proof}

Let $\chi$ be a strong $(c,c')$-coloring of $\Gnm$.
Let the colors be $\{1,\ldots,c\}$.
Let $\chi^i$ be the coloring
$$\chi^i(a,b) = \chi(a,b)+i \pmod c.$$
(During calculations mod $c$ we use $\{1,\ldots,c\}$
instead of the more traditional $\{0,\ldots,c-1\}$.)

Take $\Gnm$ with coloring $\chi$.
Place next to it $\Gnm$ with coloring $\chi^{c'}$.
Then place next to that $\Gnm$ with coloring $\chi^{2c'}$
Keep doing this until you have $\chi^{(x-1)c'}$ placed.
Table~\ref{ta:u62} is an
example using the strong $(6,2)$-coloring
of $G_{8,6}$ in Example~\ref{ex:strong}.4 to obtain a 6-coloring of $G_{8,18}$.
Since $c'=2$ and $x=3$ we will be shifting the colors
first by 2 then by 4.

\begin{table}[htbp]
\[
\begin{array}{cccccc|cccccc|cccccc}
1 & 1 & 2 & 2 & 3 & 6 	& 3 & 3 & 4 & 4 & 5 & 2 &   5 & 5 & 6 & 6 & 1 & 4 \cr
1 & 2 & 1 & 2 & 4 & 5   & 3 & 4 & 3 & 4 & 6 & 1 &   5 & 6 & 5 & 6 & 2 & 3 \cr
2 & 1 & 2 & 1 & 5 & 4   & 4 & 3 & 4 & 3 & 1 & 6 &   6 & 5 & 6 & 5 & 3 & 2 \cr
2 & 2 & 1 & 1 & 6 & 3   & 4 & 4 & 3 & 3 & 2 & 5 &   6 & 6 & 5 & 5 & 4 & 1 \cr
3 & 4 & 5 & 6 & 1 & 2   & 5 & 6 & 1 & 2 & 3 & 4 &   1 & 2 & 3 & 4 & 5 & 6 \cr
4 & 5 & 6 & 4 & 1 & 1   & 6 & 1 & 2 & 6 & 3 & 3 &   2 & 3 & 4 & 2 & 5 & 5 \cr
5 & 6 & 3 & 3 & 1 & 2   & 1 & 2 & 5 & 5 & 3 & 4 &   3 & 4 & 1 & 1 & 5 & 6 \cr
6 & 3 & 4 & 5 & 1 & 2   & 2 & 5 & 6 & 1 & 3 & 4 &   4 & 1 & 2 & 3 & 5 & 6 \cr
\end{array}
\]
\caption{Using the Strong $(6,2)$-coloring of $G_{8,6}$ to get a 6-coloring of $G_{8,18}$\label{ta:u62}}
\end{table}

We claim that the construction always creates a $c$-coloring of
$G_{m,xn}$.

We show that there is no rectangle with the two leftmost
points from the first $\Gnm$.  From this, to show that there are no
rectangles at all is just a matter of notation.

Assume that in column $i_1$ there are two points colored $R$
(in this proof $1\le R,B,G\le c$.)  We call
these \emph{the $i_1$-points}.  The points cannot form a rectangle
with any other points in $\Gnm$ since $\chi$ is a $c$-coloring
of $\Gnm$.  The $i_1$-points cannot form a rectangle with points in
columns $i_1+m$, $i_1+2m$, $\ldots$, $i_1+(c-1)m$ since the colors of
those points are $R+c'\pmod c$, $R+2c'\pmod c$, $\ldots$,
$R+(x-1)c'\pmod c$, all of which are not equal to $R$.  Is there a $j$, 
$1\le j\le x-1$ and a $i_2$, $1\le i_2\le m$ such that the $i_1$-points form
a rectangle with points in column $i_2+jm$?

Since $\chi$ is a strong $(c,c')$-coloring,
points in column $i_2$ and on the same row as the $i_1$-points
are either colored \emph{differently}, or both colors are in $[c']$.
We consider both of these cases.

\noindent
{\bf Case 1:}
In column $i_2$ the colors are $B$ and $G$ where $B\ne G$
(it is possible that $B=R$ or $G=R$ but not both).  
By the construction, the points in column 
$i_2+jm$ are colored $B+jc' \pmod c$ and $G+jc'\pmod c$.
These points are colored differently, hence they cannot form a rectangle
with the $i_1$-points.

\[
\begin{array}{ccccc|ccccc}
\cdots & i_1 & \cdots   & i_2 & \cdots & \cdots &  i_1+jm    & \cdots& i_2+jm &  \cdots \cr
\hline
\cdots & R   & \cdots   & B  & \cdots   & \cdots &  R+jc'&\cdots & B+jc'  & \cdots \cr
\cdots & R   & \cdots   & G  & \cdots   & \cdots  & R+jc'&\cdots & G+jc' & \cdots \cr
\end{array}
\]

\bigskip

\noindent
{\bf Case 2:}
In column $i_2$ the colors are both $B$.

\[
\begin{array}{ccccc|ccccc}
\cdots & i_1 & \cdots & i_2  & \cdots & \cdots  & i_1+jm    & \cdots& i_2+jm  & \cdots \cr
\hline
\cdots &  R   & \cdots   & B &  \cdots   & \cdots  & R+jc'&\cdots & B+jc' & \cdots \cr
\cdots &  R   & \cdots   & B  & \cdots   & \cdots  & R+jc'&\cdots & B+jc' & \cdots \cr
\end{array}
\]

We have $R,B\in [c']$.
By the construction, the points in column 
$i_2+jm$ are both colored $B+jc'\pmod c$.
We show that $R\not\equiv B+jc'\pmod c$.
Since $1\le j\le x-1$ we have
$$c' \le jc' \le (x-1)c'.$$
Hence
$$B+c' \le B+jc' \le B+(x-1)c'.$$
Since $B\in [c']$ we have $B+(x-1)c' \le xc'$.
Hence
$$B+c' \le B+jc' \le xc'.$$
By the definition of $x$ we have $xc' \le c$.
Since $B\in [c']$ we have $B+c' \ge c'+1$.
Hence
$$c'+1 \le B+jc' \le c.$$
Since $R\in [c']$ we have that $R\not\equiv B+jc'$.
\end{proof}

\subsection{Using Combinatorics and Strong $(c,c')$-colorings}

\begin{theorem}\label{th:cplusone}
Let $c\ge 2$. 
\begin{enumerate}
\item
There is a strong $c$-coloring of $G_{c+1,\binom{c+1}{2}}$.
\item
There is a $c$-coloring of $G_{c+1,m}$ where $m=c \binom{c+1}{2}$.
\end{enumerate}
\end{theorem}

\begin{proof}

\noindent
1) We first do an example of our construction. In the $c=5$ case
we obtain the coloring in Table~\ref{ta:s615}

\begin{table}[htbp]
\[
\begin{array}{|c|c|c|c|c|c|c|c|c|c|c|c|c|c|c|}
\hline
5 & 5 & 5 & 5 & 5 & 1 & 1 & 1 & 1 & 1 & 1 & 1 & 1 & 1 & 1 \cr
\hline
5 & 1 & 1 & 1 & 1 & 5 & 5 & 5 & 5 & 2 & 2 & 2 & 2 & 2 & 2 \cr
\hline
1 & 5 & 2 & 2 & 2 & 5 & 2 & 2 & 2 & 5 & 5 & 5 & 3 & 3 & 3 \cr
\hline
2 & 2 & 5 & 3 & 3 & 2 & 5 & 3 & 3 & 5 & 3 & 3 & 5 & 5 & 4 \cr
\hline
3 & 3 & 3 & 5 & 4 & 3 & 3 & 5 & 4 & 3 & 5 & 4 & 5 & 4 & 5 \cr 
\hline
4 & 4 & 4 & 4 & 5 & 4 & 4 & 4 & 5 & 4 & 4 & 5 & 4 & 5 & 5 \cr
\hline
\end{array}
\]
\caption{Strong 5-coloring of $G_{6,15}$\label{ta:s615}}
\end{table}

Index the columns by ${ \binom{[c+1]}{2} }$.
Color rows of column $\{x,y\}$, $x<y$,  as follows. 
\begin{enumerate}
\item
Color rows $x$ and $y$ with color $c$.
\item
On the other spots use the colors $\{1,2,3,\ldots,c-1\}$
in increasing order (the actual order does not matter).
\end{enumerate}

\bigskip

We call the coloring $\chi:\Gnm \into [c]$.
We show that there are no half-mono rectangles.
Let $RECT=\{p_1,p_2,q_1,q_2\}$ be a rectangle with
$p_1,p_2$ in column $\{x,y\}$
and $q_1,q_2$ in column $\{x',y'\}$.

If any of $p_1,p_2,q_1,q_2$ have a color in $\{1,\ldots,c-1\}$
then $RECT$ cannot be a half-mono rectangle since the colors
$\{1,\ldots,c-1\}$ only appear once in each column.

If $\chi(p_1)=\chi(p_2)=\chi(q_1)=\chi(q_2)=c$ then
$p_1$ and $p_2$ are in rows $x$ and $y$, and
$q_1$ and $q_2$ are in rows $x'$ and $y'$.
Since $RECT$ is a rectangle $\{x,y\}=\{x',y'\}$.
Hence $p_1,p_2,q_1,q_2$ are all in the same column.
This contradicts $RECT$ being a rectangle.

\smallskip

\noindent
2) This follows from Theorem~\ref{th:strong} with $c=c$ and $c'=1$,
and Part~(1) of this theorem.
\end{proof}

In order to generalize Theorem~\ref{th:cplusone} we need a lemma.
The lemma (and the examples)
is based on the Wikipedia entry on Round Robin tournaments; hence
we assume it is folklore.
We present a proof for completeness.

\begin{lemma}\label{le:rr}
Let $n\in \natt$.
\begin{enumerate}
\item
$\binom{[2n]}{2}$ can be 
partitioned into $2n-1$ sets $P_1,\ldots,P_{2n-1}$, each of size $n$,
such that each $P_i$ is itself a partition of
$[2n]$ into pairs (i.e., a perfect matching) and all of the $P_i$'s are disjoint.
\item
For each $i\in [2n+1]$ 
$\binom{[2n+1]}{2}$ can be
partitioned into $2n+1$ sets $P_1,\ldots,P_{2n+1}$, each of size $n$,
such that each $P_i$ is itself a partition of
$[2n+1]-\{i\}$ into pairs (i.e., a perfect matching) and all of the $P_i$'s are disjoint.
\end{enumerate}
\end{lemma}

\begin{proof}

\noindent
1) All arithmetic is mod $2n-1$ with two caveats:
(a) we will use $\{1,2,\ldots,2n-1\}$ rather than the more traditional $\{0,1,2,\ldots,2n-2\}$,
(b) we will use the number $2n$ and not set it equal to 1; however, $2n$
will not be involved in any calculations.
For $1\le i\le 2n-1$ we have the following partition $P_i$:

\[
\begin{array}{|c|c|c|c|c|c|c|}
2n & i+1 & i+2 & \cdots & i+n-3 & i+n-2& i+n-1 \cr
i     & i-1 & i-2 & \cdots & i-n+3 & i-n+2& i-n+1 \cr
\end{array}
\]

Formally 

$$P_i = \{2n,i\} \cup \{ \{ i+j,i-j\} \st 1\le j\le n-1 \}.$$

It is easy to see that each $P_i$ consists of disjoint pairs
and that the $P_i$'s are disjoint.

\noindent
Example: $n=4$. $1\le i\le 7$.

$P_1$
\[
\begin{array}{|c|c|c|c|}
8 & 2 & 3 & 4 \cr
1     & 7 & 6 & 5 \cr
\end{array}
\]

$P_2$
\[
\begin{array}{|c|c|c|c|}
8 & 3 & 4 & 5 \cr
2     & 1 & 7 & 6 \cr
\end{array}
\]

$P_3$
\[
\begin{array}{|c|c|c|c|}
8 & 4 & 5 & 6 \cr
3     & 2 & 1 & 7 \cr
\end{array}
\]

$P_4$
\[
\begin{array}{|c|c|c|c|}
8 & 5 & 6 & 7 \cr
4     & 3 & 2 & 1 \cr
\end{array}
\]

$P_5$
\[
\begin{array}{|c|c|c|c|}
8 & 6 & 7 & 1 \cr
5     & 4 & 3 & 2 \cr
\end{array}
\]

$P_6$
\[
\begin{array}{|c|c|c|c|}
8 & 7 & 1 & 2 \cr
6     & 5 & 4 & 3 \cr
\end{array}
\]

$P_7$
\[
\begin{array}{|c|c|c|c|}
8 & 1 & 2 & 3 \cr
7     & 6 & 5 & 4 \cr
\end{array}
\]

\bigskip

\noindent
2) We partition $[2n+1]$.
All arithmetic is be mod $2n+1$; however, we use $\{1,2,\ldots,2n+1\}$ 
rather than the more traditional $\{0,1,2,\ldots,2n\}$.
For $1\le i\le 2n+1$ we have the following partition $P_i$:

\[
\begin{array}{|c|c|c|c|c|c|c|}
i+1   & i+2 & i+3 & \cdots & i+n-3 & i+n-2& i+n \cr
i-1   & i-2 & i-3 & \cdots & i-n+3 & i-n+2& i-n \cr
\end{array}
\]

Formally 

$$P_i = \{  \{ i+j,i-j\} \st 1\le j\le n \}\}.$$

It is easy to see that each $P_i$ consists of disjoint pairs of $\{0,1,\ldots,2n\} - \{i\}$
and that the $P_i$'s are disjoint.

\noindent
Example: $n=3$. $1\le i\le 7$ and arithmetic is mod 7.

$P_1$
\[
\begin{array}{|c|c|c|}
2     & 3 & 4  \cr
7     & 6 & 5  \cr
\end{array}
\]

$P_2$
\[
\begin{array}{|c|c|c|}
3     & 4 & 5 \cr
1     & 7 & 6 \cr
\end{array}
\]

$P_3$
\[
\begin{array}{|c|c|c|}
4     & 5 & 6  \cr
2     & 1 & 7  \cr
\end{array}
\]

$P_4$
\[
\begin{array}{|c|c|c|}
5     & 6 & 7 \cr
3     & 2 & 1 \cr
\end{array}
\]

$P_5$
\[
\begin{array}{|c|c|c|}
6     & 7 & 1 \cr
4     & 3 & 2 \cr
\end{array}
\]

$P_6$
\[
\begin{array}{|c|c|c|}
7     & 1 & 2  \cr
5     & 4 & 3 \cr
\end{array}
\]

$P_7$
\[
\begin{array}{|c|c|c|}
1     & 2 & 3 \cr
6     & 5 & 4 \cr
\end{array}
\]

\bigskip

%\noindent
%Example: $n=4$. $1\le i\le 8$ and arithmetic is mod 9.
%
%$P_1$
%\[
%\begin{array}{|c|c|c|c|}
%2     & 3 & 4 & 5 \cr
%9     & 8 & 7 & 6 \cr
%\end{array}
%\]
%
%$P_2$
%\[
%\begin{array}{|c|c|c|c|}
%3     & 4 & 5 & 6 \cr
%1     & 9 & 8 & 7 \cr
%\end{array}
%\]
%
%$P_3$
%\[
%\begin{array}{|c|c|c|c|}
%4     & 5 & 6 & 7 \cr
%2     & 1 & 9 & 8 \cr
%\end{array}
%\]
%
%$P_4$
%\[
%\begin{array}{|c|c|c|c|}
%5     & 6 & 7 & 8 \cr
%3     & 1 & 9 & 8 \cr
%\end{array}
%\]
%
%$P_5$
%\[
%\begin{array}{|c|c|c|c|}
%6     & 7 & 8 & 9 \cr
%4     & 3 & 2 & 1 \cr
%\end{array}
%\]
%
%$P_6$
%\[
%\begin{array}{|c|c|c|c|}
%7     & 8 & 9 & 1 \cr
%5     & 4 & 3 & 2 \cr
%\end{array}
%\]
%
%$P_7$
%\[
%\begin{array}{|c|c|c|c|}
%8     & 9 & 1 & 2 \cr
%6     & 5 & 4 & 3 \cr
%\end{array}
%\]
%
%$P_8$
%\[
%\begin{array}{|c|c|c|c|}
%9     & 1 & 2 & 3 \cr
%7     & 6 & 5 & 4 \cr
%\end{array}
%\]
%
%
%$P_9$
%\[
%\begin{array}{|c|c|c|c|}
%1     & 2 & 3 & 4 \cr
%8     & 7 & 6 & 5 \cr
%\end{array}
%\]

\end{proof}

\begin{theorem}\label{th:cplusgen}
Let $c,c'\in \natt$ with $c\ge 2$ and $1\le c'\le c$.
\begin{enumerate}
\item
There is a strong $(c,c')$-coloring of $G_{c+c',m}$ where $m= \binom{c+c'}{2}$.
\item
There is a $c$-coloring of $G_{c+c',m'}$ where $m'=\floor{c/c'}\binom{c+c'}{2}$.
\end{enumerate}
\end{theorem}

\begin{proof}

\noindent
1)  We split into two cases.

\noindent
{\bf Case 1: $c+c'$ is even.}  Then $c+c' = 2n$ for some $n$.
Since $c'\le c$, we also have $c' \le n$.  Let $P_1,\ldots,P_{2n-1}$
be the partition of $[2n]$ of Lemma~\ref{le:rr}.1.  
Index the elements
of each $P_i$ as $p_{i,j}$ for $1\le j\le n$, that is, $P_i =
\{p_{i,1},p_{i,2},\ldots,p_{i,n}\}$.  We partition the $\binom{c+c'}{2}$ columns into
$2n-1$ parts of $n$ columns each (note that $n(2n-1) =
\binom{2n}{2}$).  We color the $j^\tth$ column in the $i^\tth$ block as
follows:
\begin{itemize}
\item The $j^\tth$ column uses color $1$ in the two rows row indexed by $p_{i,(j+1)\bmod n}$,
\item The $j^\tth$ column uses color $2$ in the two rows row indexed by $p_{i,(j+2)\bmod n}$,
\item $\vdots$
\item The $j^\tth$ column uses color $c'$ in the two rows row indexed by $p_{i,(j+c')\bmod n}$,
\item The $j^\tth$ column uses the colors $c'+1,\ldots,c$ once each to the rest of
  the elements in the column. For definiteness use them in increasing order.
\end{itemize}

We show that this yields a strong $(c',c)$-coloring.
Assume there is a half-mono rectangle. Since every color in $\{c'+1,\ldots,c\}$
only appears once in a column we have that the left and right color are
both in $[c']$. We need only prove that they are different.
Assume, by way of contradiction, that the rectangle is monochromatic and colored $d$.
Assume that one columns is column $j_1$ in part $i_1$ and the other is column $j_2$ in part $i_2$.
It is possible that $i_1=i_2$ or $j_1=j_2$ but not both.

\bigskip

\noindent
{\bf Subcase 1:} $i_1=i_2=i$, so the two columns are in the same part.
By the construction that $p_{i,(j_1+d)\bmod n} = p_{i,(j_2+d)\bmod n}$.
Since all of the $p_{ij}$'s are different this means that
$j_1\equiv j_2 \pmod n$. Since $1\le j_1,j_2\le n$ we have $j_1=j_2$.

\noindent
{\bf Subcase 2:} $i_1\ne i_2$.
By the construction this means that $p_{i_1,(j_1+d)\bmod n} = p_{i_2,(j_2+d)\bmod n}$.
Since $P_{i_1}$ and $P_{i_2}$ are disjoint this cannot happen.

\bigskip

We now give some examples of colorings.

\noindent
Example: $c'=2$ and $c=6$. $2n=c+c'=8$ so $n=4$. Note that $c+c'=8$ and $\binom{c+c'}{2}=\binom{8}{2}=28$.
Our goal is to strongly 6-color $G_{8,28}$. We use the partitions $P_1,\ldots,P_7$ in the first example in
Lemma~\ref{le:rr}. 
We first partition the 28 columns of $G_{7,28}$ into $2n-1=7$ parts of $n=4$ each:: 
$\{1,2,3,4\}$, $\{5,6,7,8\}$, $\{9,10,11,12\}$, $\{13,14,15,16\}$,
$\{17,18,19,20\}$, $\{21,22,23,24\}$, $\{25,26,27,28\}$.

We color the $i^\tth$ ($1\le i\le 9$) set of columns using $P_i$ to tell
us where to put the 1's and 2's.

We describe the coloring of the first four columns carefully. The strong 6-coloring of $G_{8,28}$ is
then in Table~\ref{ta:s828}.
then fill in the rest in a similar manner.

\[
\begin{array}{|c|c|c|c|}
\hline
p_{11} & p_{12} & p_{13} & p_{14} \cr
\hline
8     & 2 & 3 & 4 \cr
1     & 7 & 6 & 5 \cr
\hline
\end{array}
\]

Fix $i=1$, so we are looking at the $1^\sst$ part (the first four columns).
Fix $j=1$, so we are looking at the $1^\sst$ column of the $1^\sst$ part (the first column).
We put a 1 in the rows indexed by $p_{i,j+1}=p_{1,2}$.
So we put 1 in the $2^\sst$ and $7^\tth$ rows of the first column.
We put a 2 in the rows indexed by $p_{i,j+2}=p_{1,3}$.
So we put 2 in the $3^\sst$ and $6^\tth$ rows of the first column.
The rest of the rows get ${3,4,5,6}$ in increasing order.

Fix $j=2$ (the second column of the first part, so the second column).
We put a 1 in the rows indexed by $p_{i,j+1}=p_{1,3}$.
So we put 1 in the $3^\sst$ and $6^\tth$ rows of the first column.
We put a 2 in the rows indexed by $p_{i,j+2}=p_{1,4}$.
So we put 2 in the $4^\sst$ and $5^\tth$ rows of the first column.
The rest of the rows get ${3,4,5,6}$ in increasing order.

Fix $j=3$ (the third column of the first part, so the third column).
We put a 1 in the rows indexed by $p_{i,j+1}=p_{1,4}$.
So we put 1 in the $4^\sst$ and $5^\tth$ rows of the first column.
We put a 2 in the rows indexed by $p_{i,j+2}=p_{1,1}$.
So we put 2 in the $4^\sst$ and $5^\tth$ rows of the first column.
The rest of the rows get ${3,4,5,6}$ in increasing order.

Fix $j=4$ (the fourth column of the first part, so the fourth column).
We put a 1 in the rows indexed by $p_{i,j+1}=p_{1,1}$.
So we put 1 in the $1^\sst$ and $8^\tth$ rows of the first column.
We put a 2 in the rows indexed by $p_{i,j+2}=p_{1,2}$.
So we put 2 in the $2^\sst$ and $7^\tth$ rows of the first column.
The rest of the rows get ${3,4,5,6}$ in increasing order.

\begin{table}[htbp]
\[
\begin{array}{|c|c|c|c||c|c|c|c||c|c|c|c||c|c|c|c||c|c|c|c||c|c|c|c||c|c|c|c|}
\hline
               3& 3& 2& 1& 1& 3& 3& 2& 2& 1& 3& 3& 3& 2& 1& 3& 3& 2& 1& 3& 2& 1& 3& 3& 1& 3& 3& 2 \cr
\hline
               1& 4& 3& 2& 3& 4& 2& 1& 1& 3& 4& 2& 2& 1& 3& 4& 4& 2& 1& 4& 3& 2& 1& 4& 2& 1& 4& 3\cr
\hline
               2& 1& 4& 3& 1& 5& 4& 2& 3& 4& 2& 1& 1& 3& 4& 2& 2& 1& 3& 5& 4& 2& 1& 5& 3& 2& 1& 4 \cr
\hline
               4& 2& 1& 4& 2& 1& 5& 3& 1& 5& 5& 2& 4& 4& 2& 1& 1& 3& 4& 2& 2& 1& 4& 6& 4& 2& 1& 5 \cr
\hline
               5& 2& 1& 5& 4& 2& 1& 4& 2& 1& 6& 4& 1& 5& 5& 2& 5& 4& 2& 1& 1& 3& 5& 2& 2& 1& 5& 6 \cr
\hline
               2& 1& 5& 6& 5& 2& 1& 5& 4& 2& 1& 5& 2& 1& 6& 5& 1& 5& 5& 2& 5& 4& 2& 1& 1& 4& 6& 2 \cr
\hline
               1& 5& 6& 2& 6& 6& 6& 6& 5& 2& 1& 6& 5& 2& 1& 6& 2& 1& 6& 6& 1& 5& 6& 2& 5& 5& 2&1  \cr
\hline
               6& 6& 2& 1& 2& 1& 2& 1& 6& 6& 2& 1& 6& 6& 2& 1& 6& 6& 2& 1& 6& 6& 2& 1& 6& 6& 2&1  \cr
\hline
\end{array}
\]
\caption{Strong 6-coloring of $G_{8,28}$\label{ta:s828}.}
\end{table}

\noindent
{\bf Case 2: $c+c'$ is odd.}
Let $c+c' = 2n+1$ for some $n$.  Since $c' < c$, we also
have $c' \le n$.  
Let $P_1,\ldots,P_{2n+1}$ be from Lemma~\ref{le:rr}.2.
Index the elements
of each $P_i$ as $p_{i,j}$ for $1\le j\le n$, that is, $P_i =
\{p_{i,1},p_{i,2},\ldots,p_{i,n}\}$.  We partition the $\binom{c+c'}{2}$ columns into
$2n+1$ parts of $n$ columns each (note that $n(2n+1) = \binom{2n+1}{2}$).  
The description of the coloring and the proof that it works are identical to that in Case 1,
hence we omit it.

\noindent
Example: $c'=3$ and $c=4$. $2n+1=c+c'=7$ so $n=3$. Note that $c+c'=7$ 
and $\binom{c+c'}{2}=\binom{7}{2}=21$.
Our goal is to strongly 5-color $G_{7,21}$. We use the partitions $P_1,\ldots,P_7$ in the second example in
Lemma~\ref{le:rr}. 
We first partition the 21 columns of $G_{7,21}$ into $2n+1=7$ parts of $n=3$ each:: 
$\{1,2,3\}$, $\{4,5,6\}$, $\{7,8,9\}$, $\{10,11,12\}$,
$\{13,14,15\}$, $\{16,17,18\}$, $\{19,20,21\}$

We color the $i^\tth$ ($1\le i\le 7$) set of columns using $P_i$ to tell
us where to put the 1's and 2's. The final coloring is in Table~\ref{ta:s721}.

\begin{table}[htbp]
\[
\begin{array}{|c|c|c||c|c|c||c|c|c||c|c|c||c|c|c||c|c|c||c|c|c|}
\hline
	       3& 3& 3&   3& 2& 1&   1& 3& 2&   2& 1& 3&   2& 1& 3&   1& 3& 2&   3& 2& 1    \cr
\hline
               4& 2& 1&   4& 3& 3&   3& 2& 1&   1& 3& 2&   2& 1& 4&   2& 1& 3&   1& 3& 2   \cr
\hline
               1& 4& 2&   5& 2& 1&   4& 4& 3&   3& 2& 1&   1& 3& 2&   2& 1& 4&   2& 1& 3    \cr
\hline
               2& 1& 4&   1& 4& 2&   5& 2& 1&   4& 4& 4&   3& 2& 1&   1& 4& 2&   2& 1& 4    \cr
\hline
               2& 1& 5&   2& 1& 4&   1& 5& 2&   5& 2& 1&   4& 4& 5&   3& 2& 1&   1& 4& 2    \cr
\hline
               1& 5& 2&   2& 1& 5&   2& 1& 4&   1& 5& 2&   5& 2& 1&   4& 5& 5&   4& 2& 2    \cr
\hline
               5& 2& 1&   1& 5& 2&   2& 1& 5&   2& 1& 5&   1& 5& 2&   5& 2& 1&   6& 5& 5    \cr
\hline
\end{array}
\]
\caption{Strong 5-coloring of $G_{7,21}$\label{ta:s721}.}
\end{table}

\noindent
2) This follows from Theorem~\ref{th:strong} and Part~(1) of this theorem.
\end{proof}

\begin{corollary}\label{co:rr}
For all $c\ge 2$, there is
a $c$-coloring of $G_{2c,2c^2-c}.$
\end{corollary}

\subsection{Using Finite Fields and Strong $c$-colorings}

\newcommand{\pairs}{{\rm pairs}}

\begin{definition}
Let $X$ be a finite set and $q\in\natt$, $q\ge 3$.
Let $P\subseteq \binom{X}{q}$.
$$
\pairs(P)=
\{\{a_1,a_2\}\in \binom{X}{2} \st (\exists a_3,\ldots,a_q)[\{a_1,\ldots,a_q\}\in P]\}.
$$
\end{definition}

\begin{example}
Let $X=\{1,2,3,4,5,6,7,8,9\}$.
Let $q=3$.
\begin{enumerate}
\item
Let $P= \{ \{1,2,6\}, \{1,8,9\}, \{2,4,6\} \}.$
Then
$$\pairs(P)= \{ \{1,2\}, \{1,6\}, \{2,6\}, \{1,8\}, \{1,9\},
\{8,9\}, \{2,4\}, \{4,6\} \}$$
\item
Let $P= \{ \{1,2,3\}, \{4,5,6\}, \{7,8,9\} \}$.
Then
$$\pairs(P)=\{ \{1,2\}, \{1,3\}, \{2,3\}, \{4,5\}, \{4,6\}, \{5,6\},
\{7,8\}, \{7,9\}, \{8,9\} \}.$$
\end{enumerate}
\end{example}

\begin{theorem}\label{th:partitions}
Let $c,m,r\in \natt$.
Assume there exist $P_1,\ldots,P_m \subseteq\binom{[cr]}{r}$
such that the following hold.
\begin{itemize}
\item
For all $1\le j\le m$, $P_j$ is a partition of $[cr]$ into
$c$ parts of size $r$.
\item
For all $1\le j_1<j_2\le m$, $\pairs(P_{j_1})\cap\pairs(P_{j_2})=\es$.
\end{itemize}
Then
\begin{enumerate}
\item
$G_{cr,m}$ is strongly $c$-colorable.
\item
$G_{cr,cm}$ is $c$-colorable.
\end{enumerate}
\end{theorem}

\begin{proof}

%FENNER, IN THIS PROOF, WHERE DO THE 2K AND 2K-1 COME FROM? 

\noindent
1)

We define a strong $c$-coloring $\COL$ of $G_{cr,m}$ using $P_1,\ldots,P_m$.

Let  $1\le j\le m$.
Let $$P_j = \{ L_j^1,\ldots, L_j^c\}$$
where each $L_j^i $ is a subset of $r$ elements from $[cr]$.

Let $1\le i\le cr$ and $1\le j\le m$.
Since $P_j$ is a partition of $[cr]$ there exists a unique
$u$ such that $i\in L_j^u$.
Define
$$\COL(i,j)=u.$$

We show that this is a strong $c$-coloring.
Assume, by way of contradiction, that there exists
$1\le i_1<i_2\le 2k$ and $1\le j_1<j_2\le 2k-1$ such that
$\COL(i_1,j_1)=\COL(i_1,j_2)=u$ and
$\COL(i_2,j_1)=\COL(i_2,j_2)=v$.
By definition of the coloring we have
$$i_1\in L_{j_1}^u, i_1\in L_{j_2}^u,i_2\in L_{j_1}^v, i_2\in L_{j_2}^v$$
Then 
$$\{i_1,i_2\} \in \pairs(P_{j_1})\cap \pairs(P_{j_2}),$$
contradicting the second premise on the $P$'s.

\noindent
2) This follows from Part~(1) and Theorem~\ref{th:strong} with $c=c$
and $c'=1$.
\end{proof}

The Round Robin partition of Lemma~\ref{le:rr} is an example of a
partition satisfying the premises of Theorem~\ref{th:partitions}, where $c=n$,
$r=2$, and $m = 2n-1 = 2c-1$.  The next theorem yields partitions with
bigger values of $r$.

\begin{theorem}\label{th:primepower}
Let $p$ be a prime and $s,d\in\natt$.
\begin{enumerate}
\item
$G_{p^{ds},\frac{p^{ds}-1}{p-1}}$ is strongly $p^{ds-s}$-colorable.
\item
$G_{p^{ds},\frac{p^{ds}-1}{p-1}p^{ds-s}}$ is $p^{ds-s}$-colorable.
\end{enumerate}
\end{theorem}

\begin{proof}
Let $c=p^{ds-s}$, $r=p^s$, and $m=\frac{p^{ds}-1}{p^s-1}$.
We show that there exists $P_1,\ldots,P_m$
satisfying the premise of Theorem~\ref{th:partitions}.
The result follows immediately.

Let $F$ be the finite field on $p^s$ elements.
We identify $[cr]$ with the set $F^d$.

\begin{definition}~
\begin{enumerate}
\item
Let $\vec x\in F^d$, $\vec y \in F^d-\{0^d\}$.
Then
$$L_{\vec x,\vec y}=\{ \vec x + f\vec y \mid f\in F \}.$$
Sets of this form are called \emph{lines}.
Note that for all $\vec x, \vec y, a\in F$ with $a\ne 0$,
$$L_{\vec x,\vec y} = L_{\vec x,a\vec y}.$$
\item
Two lines $L_{\vec x,\vec y}$, $L_{\vec z,\vec w}$ \emph{have the same slope} if
$\vec y$ is a multiple of $\vec w$.
\end{enumerate}
\end{definition}

The following are easy to prove and well-known.
\begin{itemize}
\item
If $L$ and $L'$ are two distinct lines that have the same slope, then
$L\cap L' = \es$.
\item
If $L$ and $L'$ are two distinct lines with different slopes, then
$|L\cap L'| \le 1$.
\item
If $L$ is a line then there are exactly $r=p^s$ points on $L$.
\item
If $L$ is a line then there are exactly $c=p^{ds-s}$ lines that have the same slope as $L$ (this includes $L$ itself).
\item
There are exactly $\frac{p^{ds}-1}{p^s-1}$ different slopes.
\end{itemize}

We define $P_1,\ldots,P_m$ as follows.
\begin{enumerate}
\item
Pick a line $L$. 
Let $P_1$ be the set of lines that
have the same slope as $L$.
\item
Assume that $P_1,\ldots,P_{j-1}$ have been defined and that $j\le m$.
Let $L$ be a line that is not in $P_1\cup \cdots \cup P_{j-1}$.
Let $P_j$ be the set of all lines that
have the same slope as $L$.
\end{enumerate}

We need to show that $P_1,\ldots,P_m$ satisfies the premises
of Theorem~\ref{th:partitions}

\smallskip

\noindent
a) For all $1\le j\le m$, $P_j$ is a partition of $[cr]$ into $c$
parts of size $r$.
Let $L\in P_j$. Note that $P_j$ is the set of all lines with
the same slope as $L$. Clearly this partitions $F^{d}$ which is $[cr]$.

\smallskip

\noindent
b) For all $1\le j_1<j_2\le m$, $\pairs(P_{j_1})\cap\pairs(P_{j_2})=\es$.
Let $L_1$ be any line in $P_{j_1}$ and 
$L_2$ be any line in $P_{j_2}$. Since $|L_1\cap L_2| \le 1 < 2$ 
we have the result.

\smallskip

Note that each $P_j$ has $c=p^{ds-s}$ sets (lines) in it,
each set (line) has $r=p^s$ numbers (points),
and there are $m=\frac{p^{ds}-1}{p^s-1}$ many $P$'s.
Hence the premises of Theorem~\ref{th:partitions}
are satisfied.
\end{proof}

It is convenient to state the $s=1$, $d=2$ case of Theorem~\ref{th:primepower}.

\begin{corollary}\label{co:prime}
Let $p$ be a prime.
\begin{enumerate}
\item
There is a strong $p$-coloring of $G_{p^2,p+1}$.
\item
There is a $p$-coloring of $G_{p^2,p^2+p}$.
\end{enumerate}
\end{corollary}

\subsection{Using Finite Fields for the Square and Almost Square Case}

Can Theorem~\ref{th:primepower} be used to get that, if $c$ is a prime power,
$G_{c^2,c^2}$ is $c$-colorable. Not quite. If $d=2$ one obtains that
a grid of dimensions $p^{2s} \times \frac{p^s-1}{p-1}p^s$ is $p^d$-colorable.
Letting $c=p^s$ one gets that if $c$ is a prime power then  $c^2 \times c^{2-(1/s)+o(1)}$ is $c$-colorable.

Ken Berg and Quimey Vivas have both shown (independently) that
if $c$ is a prime power then $G_{c^2,c^2}$ is $c$-colorable.
(They both emailed us their proofs.)
Ken Berg extended this to show that if $c$ is a prime power
then $G_{c^2,c^2+c}$ is $c$-colorable. We present both proofs.
This result is orthogonal to Theorem~\ref{th:primepower} in that
there are results you can get from either that you cannot get from the other.

\begin{theorem}\label{th:primepowersq}
If $c$ is a prime power then $G_{c^2,c^2}$ is $c$-colorable.
\end{theorem}

\begin{proof}

Let $F$ be a field of $c$ elements.
We view the elements of $G_{c^2,c^2}$ as indexed by $(F\times F)\times(F\times F)$.
The colorings is

$$COL((x_1,x_2),(y_1,y_2)) = x_1y_1 + x_2 + y_2.$$

Note that all of this arithmetic takes place in the field $F$.

Assume, by way of contradiction, that there is a monochromatic rectangle.
Then there exists $w_1,w_2,x_1,x_2,y_1,y_2,z_1,z_2\in F$ such that
$((w_1,w_2),(x_1,x_2))$,
$((w_1,w_2),(y_1,y_2))$,
$((z_1,z_2),(x_1,x_2))$,
and
$((z_1,z_2),(y_1,y_2))$ are all distinct and

$$COL((w_1,w_2),(x_1,x_2))=COL((w_1,w_2),(y_1,y_2))=COL((z_1,z_2),(x_1,x_2))=COL((z_1,z_2),(y_1,y_2)).$$

Since $COL((w_1,w_2),(x_1,x_2))=COL((w_1,w_2),(y_1,y_2))$

\[
\begin{array}{rl}
w_1x_1 + w_2 + x_2 & = w_1y_1 + w_2 + y_2 \cr
w_1(x_1-y_1)  & = y_2 - x_2 \cr
\end{array}
\]

Since $COL((z_1,z_2),(x_1,x_2))=COL((z_1,z_2),(y_1,y_2))$

\[
\begin{array}{rl}
z_1x_1 + z_2 + x_2 & = z_1y_1 + z_2 + y_2 \cr
z_1(x_1-y_1)  & = y_2 - x_2 \cr
\end{array}
\]

Combining these two we get

\[
\begin{array}{rl}
w_1(x_1-y_1)  & = z_1(x_1-y_1)\cr
(w_1-z_1)(x_1-y_1)& = 0 \cr
\end{array}
\]

Since the arithmetic takes place in a field we obtain that either
$w_1=z_1$ or $x_1=y_1$.

\noindent
{\bf Case 1: $w_1=z_1$}

Since $COL((w_1,w_2),(x_1,x_2))=COL((z_1,z_2),(x_1,x_2))$. 

\[
\begin{array}{rl}
w_1x_1 + w_2 + x_2 & = z_1x_1 + z_2 + x_2 \cr
z_1x_1 + w_2 + x_2 & = z_1x_1 + z_2 + x_2 \hbox{  Since $w_1=z_1$.}\cr
w_2 & = z_2\cr
\end{array}
\]

Since $w_1=z_1$ and $w_2=z_2$ the four points are not distinct.
This is a contradiction.

\noindent
{\bf Case 2: $x_1=y_1$.} Similar to Case 1.

\end{proof}

\begin{theorem}\label{th:primepowersqplus}
If $c$ is a prime power then $G_{c^2,c^2+c}$ is $c$-colorable.
\end{theorem}

\begin{proof}

Let $F$ be a field of $c$ elements.
Let $\kstar$ be a symbol to which we assign no meaning.
We view the elements of $G_{c^2,c^2+c}$ as indexed by $(F\times F)\times(F\union \{\kstar\}\times F)$.

We describe the coloring.
Assume $x_1,x_2,y_1,y_2\in F$. 

\[
\begin{array}{rl}
COL((x_1,x_2),(y_1,y_2)) & = x_1y_1 + x_2 + y_2\cr
COL((x_1,x_2),(\kstar,y_2)) & = x_1 + y_2\cr
\end{array}
\]

Note that all of this arithmetic takes place in the field $F$.

Assume, by way of contradiction, that there is a monochromatic rectangle.
Then there exists 
$w_1,w_2,x_1,x_2,y_1,y_2,z_1,z_2$ such that
\begin{itemize}
\item
$w_1,w_2,x_2,y_2,z_1,z_2\in F$
\item
$x_1,y_1\in F\cup \{\kstar\}$.
\item
$((w_1,w_2),(x_1,x_2))$,
$((w_1,w_2),(y_1,y_2))$,
$((z_1,z_2),(x_1,x_2))$,
$((z_1,z_2),(y_1,y_2))$ are all distinct.
\item 
$((w_1,w_2),(x_1,x_2))$,
$((w_1,w_2),(y_1,y_2))$,
$((z_1,z_2),(x_1,x_2))$,
$((z_1,z_2),(y_1,y_2))$ are all the same color.
\end{itemize}

By the proof of Theorem~\ref{th:primepowersq} at least one of $x_1,y_1$ is $\kstar$.
We can assume $x_1=\kstar$.
There are two cases.

\noindent
{\bf Case 1:} $y_1=\kstar$.
Since 

$$COL((w_1,w_2),(\kstar,x_2))=COL((w_1,w_2),(\kstar,y_2))$$

we have

$$w_1 + x_2 = w_1 + y_2$$

so $x_2=y_2$. Hence $(x_1,x_2)=(y_1,y_2)$ so the points are not distinct.

\smallskip

\noindent
{\bf Case 2:} $y_1\ne \kstar$.
Since 

$$COL((w_1,w_2),(\kstar,x_2))=COL((z_1,z_2),(\kstar,x_2))$$

we have

$$w_1+x_2 = z_1 + x_2$$

so $w_1=z_1$.
Since 

$$COL((w_1,w_2),(y_1,y_2))=COL((z_1,z_2),(y_1,y_2))$$

we have

$$w_1y_1 + w_2 + y_2 = z_1y_1 + z_2 + y_2.$$

Since $w_1=z_1$ we have $w_2 = z_2$.
Hence we have $(w_1,w_2)=(z_1,z_2)$ so the points are not distinct.

\end{proof}

\section{Bounds on the Sizes of Obstruction Sets}\label{se:bounds}

\subsection{An Upper Bound}

Using the uncolorability bounds, we can obtain an upper-bound on the
size of a $c$-colorable grid.

\begin{theorem}\label{th:csq}
\label{obsupperbound}
For all $c > 0$, $G_{c^2+c, c^2+c}$ is not $c$-colorable.
\end{theorem}

\begin{proof}
We apply Corollary \ref{uncolor3} with $m=c^2+c$ and $n=c^2+c$.  Note that

\begin{align*}
\bigceiling{\frac{nm}{c}} & = \bigceiling{\frac{(c^2+c)(c^2+c)}{c}} \\
&=  (c+1)(c^2+c).
\end{align*}

\noindent Letting $q = c+1$ and $r = 0$, we have

\begin{align*}
\frac{m(m-1)-2qr}{q(q-1)} & = \frac{(c^2+c)(c^2+c-1)}{(c+1)c}\\
& = c^2+c-1 \\
& < c^2+c\\
& = n.
\end{align*}\end{proof}

Using this, we can obtain an upper-bound on the size of an obstruction set.

\begin{theorem}\label{th:obsc}
If $c > 0$, then $|\OBS_c|\le 2c^2$.
\end{theorem}

\begin{proof}
For each $r$, there can be at most one element of $\OBS_c$ of the form
$G_{r,n}$.  Likewise, there can be at most one element of $\OBS_c$ of the
form $G_{n,r}$.  If $r \leq c$ then for all $n$, $G_{r,n}$ and
$G_{n,r}$ are trivially $c$-colorable and are, therefore, not an element of $\OBS_c$.
Theorem \ref{obsupperbound} shows that for all $n,m >
c^2+c$, $\Gnm$ is not an element of $\OBS_c$.  It follows that there can be at
most two elements of $\OBS_c$ for each integer $r$ where $c < r \leq
c^2+c$.  Therefore $|\OBS_c|\le 2c^2$.
\end{proof}

\begin{note}
We will later see that $|OBS_2|=3$, $|OBS_3|=8$, and $|OBS_4|=16$.
Based on this (scant) evidence the bound of $2c^2$ looks like its too large.
\end{note}

\subsection{A Lower Bound}

To get a lower bound on $|\OBS_c|$, we will combine
Corollary~\ref{uncolor2a} and Theorem~\ref{th:cplusgen}(2) with the
following lemma:

\begin{lemma}\label{le:finding-obs}
Suppose that $G_{m_1,n}$ is $c$-colorable and $G_{m_2,n}$ is not
$c$-colorable.  Then there exists $n,m$ such that 
$m_1 < x \le m_2$, $y \le n$, and 
a grid $G_{x,y}\in\OBS_c$.
\end{lemma}

\begin{proof}
Given $n$, let $x$ be the least integer such that $G_{x,n}$ is not $c$-colorable.
Clearly, $m_1 < x \le m_2$.  Now given $x$ as above, let $y$ be least
such that $G_{x,y}$ is not $c$-colorable.  Clearly, $y \le n$ and
$G_{x,y}\in\OBS_c$.
\end{proof}

\begin{theorem}
$|\OBS_c| \ge 2\sqrt{c}(1-o(1))$.
\end{theorem}

\begin{proof}
For any $c \ge 2$ and any $1\le c' \le c$ we can summarize
Corollary~\ref{uncolor2a} and Theorem~\ref{th:cplusgen}(2) as follows:
\[ G_{c+c',n} \mbox{ is } \left\{ \begin{array}{ll}
\mbox{$c$-colorable } & \mbox{if $n \le \floor{\frac{c}{c'}}\binom{c+c'}{2}$,}
\\
\mbox{not $c$-colorable } & \mbox{if $n > \frac{c}{c'}\binom{c+c'}{2}$.}
\end{array} \right. \]
(We won't use the fact here, but note that this is tight if
$c'$ divides $c$.)

Suppose $c'>1$ and
\begin{equation}\label{eq:obs}
\frac{c}{c'}\binom{c+c'}{2} < \floor{\frac{c}{c'-1}}\binom{c+c'-1}{2}.
\end{equation}
Then letting $n := \floor{\frac{c}{c'-1}}\binom{c+c'-1}{2}$, we see that
$G_{c+c'-1,n}$ is $c$-colorable, but $G_{c+c',n}$ is not.  Then by
Lemma~\ref{le:finding-obs}, there is a grid $G_{c+c',y} \in \OBS_c$
for some $y$.  So there are at least as many elements of $\OBS_c$ as
there are values of $c'$ satisfying Inequality~(\ref{eq:obs})---actually
twice as many, because $\Gnm\in \OBS_c$ iff $\Gmn\in OBS_c$.

Fix any real $\eps > 0$.  Clearly, Inequality~(\ref{eq:obs}) holds provided
\[ \frac{c}{c'}\binom{c+c'}{2} \le
\left(\frac{c}{c'-1}-1\right)\binom{c+c'-1}{2}. \]
A rather tedious calculation reveals that if $2 \le c' \le
(1-\eps)\sqrt{c}$, then this latter inequality holds for all large
enough $c$.  Including the grid $G_{c+1,n}\in\OBS_c$ where $n =
c\binom{c+1}{2}+1$, we then get $|\OBS_c| \ge \floor{(1-\eps)\sqrt{c}}$
for all large enough $c$, and since $\eps$ was arbitrary, we therefore
have $|\OBS_c| \ge \sqrt{c}(1-o(1))$.

To double the count, we notice that $c+c' \le
\floor{\frac{c}{c'}}\binom{c+c'}{2}$, hence $G_{c+c',c+c'}$ is
$c$-colorable by Theorem~\ref{th:cplusgen}(2).  This means that
$G_{c+c',y}\in\OBS_c$ for some $y > c+c'$, and so we can count
$G_{y,c+c'}\in\OBS_c$ as well without counting any grids twice.
\end{proof}

\section{Which Grids Can be 2-Colored?}\label{se:col2}

\begin{theorem}\label{th:col2}~
\begin{enumerate}
\item
$G_{7,3}$ and $G_{3,7}$ are not 2-colorable
\item
$G_{5,5}$ is not 2-colorable.
\item
$G_{7,2}$ and $G_{2,7}$ are 2-colorable (this is trivial).
\item
$G_{6,4}$ and $G_{4,6}$ are 2-colorable.
\end{enumerate}
\end{theorem}

\begin{proof}

We only consider grids of the form $\Gnm$ where $n\ge m$.

\noindent
1,2)

In Table~\ref{ta:uncol2} we show that $G_{7,3}$ and $G_{5,5}$ are not 2-colorable.
For each $(n,m)$ we use either Corollary~\ref{uncolor2} or \ref{uncolor3}.
In the table we give, for each $(n,m)$, the value of $\ceil{\frac{nm}{2}}$,
the $q,r$ such that $\ceil{\frac{nm}{2}}=qn+r$ with $0\le r\le n-1$,
which corollary we use (Use),  the premise of the corollary (Prem), and the arithmetic showing the premise is true 
(Arith).

\begin{table}[htbp]
\centering
\begin{tabular}{c c c c c c c c }
$m$ & $n$ & $\lceil \frac{nm}{2} \rceil$ & $q$ & $r$ & Use & Prem & Arith  \\ [0.5ex]
\hline\hline
3 & 7 & 11 & 1 & 4 & Cor~\ref{uncolor2} & $\qone$  & $3 < 4 \le 7$ \\
5 & 5 & 13 & 2 & 3 & Cor~\ref{uncolor3} & $\qnotone$  & $4 < 5$ \\
\hline
\end{tabular}
\caption{$(m,n)$ such that $G_{m,n}$ is not 2-colorable \label{ta:uncol2}}
\end{table}

\smallskip

\noindent
3) $G_{7,2}$ is clearly 2-colorable.

\smallskip

\noindent
4) $G_{6,4}$ is 2-colorable 
by Corollary~\ref{co:prime} with $p=2$.
We present that coloring  in Table~\ref{ta:col46} below.

\begin{table}[htbp]
 \[
 \begin{array}{|c|c|c|c|c|c|}
\hline
  R&    R &  R &  B &  B &  B \cr
 \hline
  R&    B &  B &  B &  R &  R \cr
 \hline
  B&    R &  B &  R &  B &  R  \cr
 \hline
  B&    B &  R &  R &  R &  B  \cr
 \hline
 \end{array}
 \]
\caption{2-Coloring of $G_{4,6}$\label{ta:col46}}
\end{table}

\end{proof}

\begin{theorem}\label{th:obs2}
$\OBS_2 =\{G_{7,3},G_{5,5},G_{3,7}\}$.
\end{theorem}

\begin{proof}

$G_{7,3}$ is not $2$-colorable by Theorem~\ref{th:col2}.
$G_{6,3}$ is $2$-colorable by Theorem~\ref{th:col2}.
$G_{7,2}$ is $2$-colorable by Theorem~\ref{th:col2}.
Hence $G_{7,3}\in \OBS_2$.
The proof for $G_{3,7}$ is similar.

$G_{5,5}$ is not $2$-colorable by Theorem~\ref{th:col2}.
$G_{5,4}$ and $G_{4,5}$ are $2$-colorable by Theorem~\ref{th:col2}.
Hence $G_{5,5}\in OBS_2$.

Table~\ref{ta:col2} indicates exactly which grids are 2-colorable.
The entry for $(n,m)$ is $C$ if $\Gnm$ is 2-colorable, and
$N$ if $\Gnm$ is not 2-colorable.
From this Table one easily sees that the grids listed in this theorem
are the only elements of $\OBS_2$.

\begin{table}[htbp]
 \[
 \begin{array}{|c|c|c|c|c|c|c|c|}
 \hline
      & 2 &  3 &  4 &  5 &  6 &  7 & 8  \cr
\hline
  2&    C &  C &  C &  C &  C  & C & C \cr
 \hline
  3&    C &  C &  C &  C &  C  & N & N \cr
 \hline
  4&    C &  C &  C &  C &  C  & N & N  \cr
 \hline
  5&    C &  C &  C &  N &  N  & N & N  \cr
 \hline
  6&    C &  C  & C  & N  & N  & N & N  \cr
 \hline
  7&    C &  N  & N  & N  & N  & N & N \cr
 \hline
  8&    C &  N  & N  & N  & N  & N & N \cr
 \hline
 \end{array}
 \]
\caption{2-Colorable Grids ($C$) and non 2-Colorable Grids ($N$)}\label{ta:col2}
\end{table}

\end{proof}

\section{Which Grids Can be 3-Colored?}\label{se:col3}

\begin{theorem}\label{th:col3}~
\begin{enumerate}
\item
$G_{19,4}$ and $G_{4,19}$ are  not 3-colorable.
\item
$G_{16,5}$ and $G_{5,16}$ are not 3-colorable.
\item
$G_{13,7}$ and $G_{7,13}$ are not 3-colorable.
\item
$G_{12,10}$ and $G_{10,12}$ are not 3-colorable.
\item
$G_{11,11}$ is not 3-colorable.
\item
$G_{19,3}$ and $G_{3,19}$ are 3-colorable (this is trivial).
\item
$G_{18,4}$ and $G_{4,18}$ are  3-colorable.
\item
$G_{15,6}$ and $G_{6,15}$ are 3-colorable.
\item
$G_{12,9}$ and $G_{9,12}$ are 3-colorable.
\end{enumerate}
\end{theorem}

\begin{proof}

We just consider the grids $G_{n,m}$ were $n\ge m$.

\noindent
$1,2,3,4,5$)

In Table~\ref{ta:uncol3} we show that several grids are not 3-colorable.
For each $(n,m)$ we use either Corollary~\ref{uncolor2} or \ref{uncolor3}.
In the table we give, for each $(n,m)$, the value of $\ceil{\frac{nm}{3}}$,
the $q,r$ such that $\ceil{\frac{nm}{3}}=qn+r$ with $0\le r\le n-1$,
which corollary we use (Use),  the premise of the corollary (Prem), and the arithmetic showing the premise is true 
(Arith).

\begin{table}[htbp]
\centering
\begin{tabular}{c c c c c c c c }
$m$ & $n$ & $\lceil \frac{nm}{3} \rceil$ & $q$ & $r$ & Use & Prem & Arith  \\ [0.5ex]
\hline\hline
4  & 19 & 26 & 1 &  7 & Cor~\ref{uncolor2} & $\qone$     & $6 < 7 \le 19$ \\
5  & 16 & 27 & 1 & 11 & Cor~\ref{uncolor2} & $\qone$     & $10< 11 \le 16$ \\
7  & 13 & 31 & 2 &  5 & Cor~\ref{uncolor3} & $\qnotone$  & $11 < 13$\\
10 & 12 & 40 & 3 &  4 & Cor~\ref{uncolor3} & $\qnotone$  & $11 < 12$\\
11 & 11 & 41 & 3 &  8 & Cor~\ref{uncolor3} & $\qnotone$  & $10 \frac{1}{3} < 11$\\
\hline
\end{tabular}
\caption{$(m,n)$ such that $G_{m,n}$ is not 3-colorable \label{ta:uncol3}}
\end{table}

\bigskip

\noindent
$6$) $G_{19,3}$ is clearly 3-colorable.

\bigskip

\noindent
$7$) 
$G_{18,4}$ is 3-colorable by Theorem~\ref{th:cplusone} with $c=3$.

\bigskip

\noindent
$8$) $G_{15,6}$ is 3-colorable by Corollary~\ref{co:rr} with $c=3$.

\bigskip

\noindent
$9$) 
$G_{12,9}$ is 3-colorable by Corollary~\ref{co:prime} with $p=3$.
\end{proof}

\begin{theorem}\label{th:10x10}
$G_{10,10}$ is 3-colorable.
\end{theorem}

\begin{proof}
The 3-coloring is in Table~\ref{ta:1010}.

\begin{table}[htpb]
\[
\begin{array}{|c|c|c|c|c|c|c|c|c|c|}
\hline
 1 & 1 & 1 & 1 & 2 & 2 & 3 & 3 & 2 & 3 \cr
\hline
 1 & 2 & 2 & 3 & 1 & 1 & 1 & 3 & 3 & 2 \cr
\hline
 3 & 1 & 2 & 3 & 1 & 2 & 2 & 1 & 1 & 3 \cr
\hline
 3 & 2 & 1 & 2 & 2 & 1 & 3 & 1 & 3 & 1 \cr
\hline
 1 & 2 & 3 & 3 & 3 & 2 & 3 & 2 & 1 & 1 \cr
\hline
 3 & 1 & 2 & 2 & 3 & 3 & 1 & 2 & 2 & 1 \cr
\hline
 2 & 3 & 1 & 2 & 3 & 2 & 1 & 3 & 1 & 2 \cr
\hline
 2 & 2 & 3 & 1 & 1 & 3 & 2 & 3 & 2 & 1 \cr
\hline
 3 & 3 & 3 & 1 & 2 & 1 & 2 & 2 & 1 & 2 \cr
\hline
 2 & 3 & 2 & 1 & 2 & 3 & 1 & 1 & 3 & 3 \cr
\hline
 \end{array}
\]
\caption{3-Coloring of $G_{10,10}$\label{ta:1010}.}
\end{table}
\end{proof}

\begin{note}
We found the coloring in Theorem~\ref{th:10x10} by the following steps.
\begin{itemize}
\item
We found a size 34 rectangle free subset of $G_{10,10}$ (by hand).
Frankly we were trying to prove there was no such rectangle free set and hence
$G_{10,10}$ would not be 3-colorable.
\item
We used the rectangle free set for color 1 and completed the coloring 
with a simple computer program.  
\end{itemize}
It is an open problem to find a general theorem that has a corollary that 
$G_{10,10}$ is 3-colorable.
\end{note}

\begin{theorem}\label{th:11x10}
If $A\subseteq G_{11,10}$ and
$A$ is rectangle-free then $|A|\le 36=\ceil{\frac{11\cdot 10}{3}}-1$.
Hence $G_{11,10}$ is not 3-colorable.
\end{theorem}

\begin{proof}

We divide the proof into cases.
Every case will either conclude that
$|A|\le 36$ or $A$ cannot exist.

For $1\le j\le 10$ let $x_j$ be the number of elements of
$A$ in column $j$.
We assume 

$$x_1\ge \cdots \ge x_{10}.$$

\begin{enumerate}
\item
$5\le x_1\le 11$.

By Lemma~\ref{le:maxrf1} with $x=5$, $n=11$, $m=10$ we have

$$|A| \le x + m-1 +  \maxrf {n-x}{m-1} \le 5 + 10 -1 + \maxrf {11-5} {10-1} \le 14+\maxrf{6}{9}.$$

By Lemma~\ref{le:values} we have $\maxrf{6}{9}=21$. Hence

$$|A|\le14 + 21 = 35\le 36.$$

\item
There exists $k$, $0\le k\le 6$, such that
$x_1=\cdots = x_k=4$ and $x_{k+1}\le 3$.
Then

$$|A| = \sum_{j=1}^{10} x_j = (\sum_{j=1}^{k} x_j) + (\sum_{j=k+1}^{10} x_j) \le 4k + 3(10-k)= 30+k$$

Since $k\le 6$ this quantity is $\le 30+6=36.$
Hence $|A|\le 36$.

\item
$x_1=\cdots=x_7=4$. 
Let $G'$ be the grid restricted to the first 7 
columns.
Let $B$ be $A$ restricted to $G'$.
\begin{enumerate}
\item
There exists $1\le j_1<j_2<j_3\le 7$ such that
$$|C_{j_1}\cap C_{j_2} \cap C_{j_3}|=1.$$ 
By renumbering we can assume that 
$$|C_1\cap C_2\cap C_3|=1$$
and that the intersection is in row 10.
The following picture summarizes our knowledge.
We use $R$ to denote where an element of $A$ is.

\[
\begin{array}{|c|c|c|c|c|c|c|c|}
\hline
  & 1 & 2 & 3 & 4 & 5 & 6 & 7 \cr
\hline
1 & R &   &   &   &   &   &   \cr
\hline
2 & R &   &   &   &   &   &   \cr
\hline
3 & R &   &   &   &   &   &   \cr
\hline
4 &   & R &   &   &   &   &   \cr
\hline
5 &   & R &   &   &   &   &   \cr
\hline
6 &   & R &   &   &   &   &   \cr
\hline
7 &   &   & R &   &   &   &   \cr
\hline
8 &   &   & R &   &   &   &   \cr
\hline
9 &   &   & R &   &   &   &   \cr
\hline
10 & R & R & R &   &   &   &   \cr
\hline
11 &   &   &   &   &   &   &   \cr
\hline
\end{array}
\]

Let $4\le j\le 7$.
In column $j$ there can be at most 1 $R$ 
in rows 1,2,3, at most 1 $R$ in rows 4,5,6, at
most 1 $R$ in row 7,8,9,10.
Hence, since $x_j=4$, the $j$th column has an $R$ in the 11th row. 
Also note that there must be exactly 1 R among rows 1,2,3,
exactly 1 R among rows 4,5,6, and exactly one R among
rows 7,8,9,10.

One can easily show that after a permutation of the rows
we must have the following in the first 6 columns:

\[
\begin{array}{|c|c|c|c|c|c|c|c|}
\hline
  & 1 & 2 & 3 & 4 & 5 & 6 & 7 \cr
\hline
1 & R &   &   & R &   &   &   \cr
\hline
2 & R &   &   &   & R &   &   \cr
\hline
3 & R &   &   &   &   & R &   \cr
\hline
4 &   & R &   &R  &   &   &   \cr
\hline
5 &   & R &   &   &R  &   &   \cr
\hline
6 &   & R &   &   &   &R  &   \cr
\hline
7 &   &   & R & R &   &   &   \cr
\hline
8 &   &   & R &   &R  &   &   \cr
\hline
9 &   &   & R &   &   & R &   \cr
\hline
10 & R & R & R &   &   &   &   \cr
\hline
11 &   &   &   & R & R & R & R \cr
\hline
\end{array}
\]

It is easy to see that if an $R$ is placed anywhere in column 7
then a rectangle is formed.

\item
There exists $1\le j_1<j_2<j_3\le 7$ such that
$|C_{j_1}\cap C_{j_2}|=|C_{j_1}\cap C_{j_2}|=|C_{j_2}\cap C_{j_3}|=1$.
We can assume that for all sets of three columns their intersection is $\es$
(else we would be in Case a).
We can assume that $j_1=1$, $j_2=2$, $j_3=3$ and that the first
three columns are as in the picture below.

\[
\begin{array}{|c|c|c|c|c|c|c|c|}
\hline
  & 1 & 2 & 3 & 4 & 5 & 6 & 7 \cr
\hline
1 & R &   &   &   &   &   &   \cr
\hline
2 & R &   &   &   &   &   &   \cr
\hline
3 &   & R &   &   &   &   &   \cr
\hline
4 &   & R &   &   &   &   &   \cr
\hline
5 &   &   & R &   &   &   &   \cr
\hline
6 &   &   & R &   &   &   &   \cr
\hline
7 & R & R &   &   &   &   &   \cr
\hline
8 & R &   & R &   &   &   &   \cr
\hline
9 &   & R & R &   &   &   &   \cr
\hline
10 &   &   &   &   &   &   &   \cr
\hline
11 &   &   &   &   &   &   &   \cr
\hline
\end{array}
\]

Since no three columns intersect there can be no $R$'s in the
7th, 8th, or 9th row of columns 4,5,6,7.
In later pictures we will use $X$ to denote that an $R$ cannot be in that space.

There are essentially five ways that the columns 4,5,6,7 and rows 10,11
can be arranged. Here are all of them:

\[
\begin{array}{|c|c|c|c|c|}
\hline
  &  4 & 5 & 6 & 7 \cr
\hline
10 &   &   &   &      \cr
\hline
11 & R & R & R & R \cr
\hline
\end{array}
\]

\[
\begin{array}{|c|c|c|c|c|}
\hline
  &  4 & 5 & 6 & 7 \cr
\hline
10 &   &   &   & R    \cr
\hline
11 & R & R & R &   \cr
\hline
\end{array}
\]

\[
\begin{array}{|c|c|c|c|c|}
\hline
  &  4 & 5 & 6 & 7 \cr
\hline
10 &   &   & R & R    \cr
\hline
11 & R & R &   &   \cr
\hline
\end{array}
\]

\[
\begin{array}{|c|c|c|c|c|}
\hline
  &  4 & 5 & 6 & 7 \cr
\hline
10 &   &   &   & R    \cr
\hline
11 & R & R & R & R \cr
\hline
\end{array}
\]

\[
\begin{array}{|c|c|c|c|c|}
\hline
  &  4 & 5 & 6 & 7 \cr
\hline
10 &   &   & R & R    \cr
\hline
11 & R & R & R &   \cr
\hline
\end{array}
\]

The third one is the hardest to analyze. Hence we analyze
that one and leave the rest to the reader.
We can assume the following picture happens.

\[
\begin{array}{|c|c|c|c|c|c|c|c|}
\hline
  & 1 & 2 & 3 & 4 & 5 & 6 & 7 \cr
\hline
1 & R &   &   &   &   &   &   \cr
\hline
2 & R &   &   &   &   &   &   \cr
\hline
3 &   & R &   &   &   &   &   \cr
\hline
4 &   & R &   &   &   &   &   \cr
\hline
5 &   &   & R &   &   &   &   \cr
\hline
6 &   &   & R &   &   &   &   \cr
\hline
7 & R & R &   & X &X  & X & X \cr
\hline
8 & R &   & R & X &X  & X & X \cr
\hline
9 &   & R & R & X & X & X &X  \cr
\hline
10 &   &   &   &   &   & R & R \cr
\hline
11 &   &   &   & R & R &   &   \cr
\hline
\end{array}
\]

In columns 4,5,6,7 there must be exactly one $R$ from row 1 or 2,
exactly one $R$ from row 3 or 4, and exactly one $R$ from row 5 or 6.
If we look just at columns 4 and 5, and permute the rows as needed,
we can assume that the following picture happens:

\[
\begin{array}{|c|c|c|c|c|c|c|c|}
\hline
  & 1 & 2 & 3 & 4 & 5 & 6 & 7 \cr
\hline
1 & R &   &   & R &   &   &   \cr
\hline
2 & R &   &   &   & R &   &   \cr
\hline
3 &   & R &   & R &   &   &   \cr
\hline
4 &   & R &   &   & R &   &   \cr
\hline
5 &   &   & R & R &   &   &   \cr
\hline
6 &   &   & R &   & R &   &   \cr
\hline
7 & R & R &   &X  &X  &X  &X  \cr
\hline
8 & R &   & R &X  &X  &X  &X  \cr
\hline
9 &   & R & R &X  &X  &X  &X  \cr
\hline
10 &   &   &   &   &   & R & R \cr
\hline
11 &   &   &   & R & R &   &   \cr
\hline
\end{array}
\]

Either there is an $R$ at both (Row1,Col6) and (Row2,Col7) or
there is an $R$ at both (Row1,Col7) and (Row2,Col6). 
We call the first one {\it slanting NW-SE} and the former {\it slanting NE-SW}.
Similar conditions apply for Rows 3 and 4, and Rows 5 and 6.
Two of the pairs of rows must have the same slant.
We can assume that Rows 1,2 and Rows 3,4 both slant NW-SE
(the other cases are similar).
We can assume that the following picture happens:

\[
\begin{array}{|c|c|c|c|c|c|c|c|}
\hline
  & 1 & 2 & 3 & 4 & 5 & 6 & 7 \cr
\hline
1 & R &   &   & R &   & R &   \cr
\hline
2 & R &   &   &   & R &   & R \cr
\hline
3 &   & R &   & R &   & R &   \cr
\hline
4 &   & R &   &   & R &   & R \cr
\hline
5 &   &   & R & R &   &   &   \cr
\hline
6 &   &   & R &   & R &   &   \cr
\hline
7 & R & R &   & X & X & X & X \cr
\hline
8 & R &   & R & X &X  &X  &X  \cr
\hline
9 &   & R & R & X & X & X &X  \cr
\hline
10 &   &   &   &   &   & R & R \cr
\hline
11 &   &   &   & R & R &   &   \cr
\hline
\end{array}
\]

Clearly a rectangle is formed.

\item
There exists $1\le j_1<j_2<j_3\le 7$ such that
$|C_{j_1}\cap C_{j_2}|=|C_{j_1}\cap C_{j_3}|=1$
but $|C_{j_2}\cap C_{j_3}|=0$.
We can assume that for all sets of three columns their intersection is $\es$
(else we would be in Case a).
We can assume that $j_1=1$, $j_2=2$, $j_3=3$ and that the first
three columns are as in the picture below.

\[
\begin{array}{|c|c|c|c|c|c|c|c|}
\hline
  & 1 & 2 & 3 & 4 & 5 & 6 & 7 \cr
\hline
1 & R &   &   &   &   &   &   \cr
\hline
2 & R &   &   &   &   &   &   \cr
\hline
3 & R & R &   &   &   &   &   \cr
\hline
4 & R &   & R &   &   &   &   \cr
\hline
5 &   & R &   &   &   &   &   \cr
\hline
6 &   & R &   &   &   &   &   \cr
\hline
7 &   & R &   &   &   &   &   \cr
\hline
8 &   &   & R &   &   &   &   \cr
\hline
9 &   &   & R &   &   &   &   \cr
\hline
10 &   &   & R &   &   &   &   \cr
\hline
11 &   &   &   &   &   &   &   \cr
\hline
\end{array}
\]

The proof is similar to the proof of case 3a.

\item
There exists $1\le j_1<j_2<j_3\le 7$ such that
$|C_{j_1}\cap C_{j_2}|=|C_{j_1}\cap C_{j_3}|=|C_{j_1}\cap C_{j_3}|=0$.
We can assume that $C_{j_1}=C_1$ and has rows $1,2,3,4$,
                   $C_{j_2}=C_2$ and has rows $5,6,7,8$, and
                   $C_{j_3}=C_3$ and has rows $9,10,11,12$.
Too bad we only have 11 rows!
\end{enumerate}
\end{enumerate}

\end{proof}

\begin{theorem}\label{th:obs3}
$$\OBS_3=
\{G_{19,4}, G_{16,5}, G_{13,7}, G_{11,10}, G_{10,11}, G_{7,13}, G_{5,16}, G_{4,19} 
\}.$$
\end{theorem}

\begin{proof}

We only deal with $G_{n,m}$ where $n\ge m$.
We show that, for all the grids $G_{n,m}$ listed where $n\ge m$,
$G_{n,m}$ is not 3-colorable but $G_{m-1,m}$ and $G_{m,m-1}$ are 3-colorable.

\begin{enumerate}
\item
$G_{19,4}$ is not 3-colorable by Theorem~\ref{th:col3}.
$G_{18,4}$ and $G_{19,3}$ are 3-colorable by Theorem~\ref{th:col3}.
\item
$G_{16,5}$ is not 3-colorable by Theorem~\ref{th:col3}.
$G_{15,5}$ and $G_{16,4}$ are 3-colorable by Theorem~\ref{th:col3}.
\item
$G_{13,7}$ is not 3-colorable by Theorem~\ref{th:col3}.
$G_{12,7}$ and $G_{13,6}$ are 3-colorable by Theorem~\ref{th:col3}.
\item
$G_{11,10}$ is not 3-colorable by Theorem~\ref{th:11x10}.
$G_{10,10}$ is 3-colorable by Theorem~\ref{th:10x10}.
$G_{11,9}$ is 3-colorable by Theorem~\ref{th:col3}.
\end{enumerate}

Table~\ref{ta:col3} indicates exactly which grids are 3-colorable.
The entry for $(n,m)$ is $C$ if $\Gnm$ is 3-colorable, and
$N$ if $\Gnm$ is not 3-colorable.
From this table one easily sees that the grids listed in this theorem
are the only elements of $\OBS_3$.

\begin{table}[htpb]
 \[
 \begin{array}{|c|c|c|c|c|c|c|c|c|c|c|c|c|c|c|c|c|c|c|}
 \hline
   & 03 & 04 & 05 & 06 & 07 & 08 & 09 & 10 & 11 & 12 & 13 & 14 & 15 & 16 & 17 & 18 & 19 & 20 \cr
 \hline
  3& C  & C  & C  & C  & C  & C  & C  & C  & C  & C  & C  & C  & C  & C  & C  & C  & C & C \cr
 \hline
  4& C  & C  & C  & C  & C  & C  & C  & C  & C  & C  & C  & C  & C  & C  & C  & C  & N & N \cr
 \hline
  5& C  & C  & C  & C  & C  & C  & C  & C  & C  & C  & C  & C  & C  & N  & N  & N  & N & N \cr
 \hline
  6& C  & C  & C  & C  & C  & C  & C  & C  & C  & C  & C  & C  & C  & N  & N  & N  & N & N \cr
 \hline
  7& C  & C  & C  & C  & C  & C  & C  & C  & C  & C  & N  & N  & N  & N  & N  & N  & N & N \cr
 \hline
  8& C  & C  & C  & C  & C  & C  & C  & C  & C  & C  & N  & N  & N  & N  & N  & N  & N & N \cr
 \hline
  9& C  & C  & C  & C  & C  & C  & C  & C  & C  & C  & N  & N  & N  & N  & N  & N  & N & N \cr
 \hline
 10& C  & C  & C  & C  & C  & C  & C  & C  & N  & N  & N  & N  & N  & N  & N  & N  & N & N \cr
 \hline
 11& C  & C  & C  & C  & C  & C  & C  & N  & N  & N  & N  & N  & N  & N  & N  & N  & N & N \cr
 \hline
 12& C  & C  & C  & C  & C  & C  & C  & N  & N  & N  & N  & N  & N  & N  & N  & N  & N & N \cr
 \hline
 13& C  & C  & C  & C  & N  & N  & N  & N  & N  & N  & N  & N  & N  & N  & N  & N  & N & N \cr
 \hline
 14& C  & C  & C  & C  & N  & N  & N  & N  & N  & N  & N  & N  & N  & N  & N  & N  & N & N \cr
 \hline
 15& C  & C  & C  & C  & N  & N  & N  & N  & N  & N  & N  & N  & N  & N  & N  & N  & N & N \cr
 \hline
 16& C  & C  & N  & N  & N  & N  & N  & N  & N  & N  & N  & N  & N  & N  & N  & N  & N & N \cr
 \hline
 17& C  & C  & N  & N  & N  & N  & N  & N  & N  & N  & N  & N  & N  & N  & N  & N  & N & N \cr
 \hline
 18& C  & C  & N  & N  & N  & N  & N  & N  & N  & N  & N  & N  & N  & N  & N  & N  & N & N \cr
 \hline
 19& C  & N  & N  & N  & N  & N  & N  & N  & N  & N  & N  & N  & N  & N  & N  & N  & N & N \cr
 \hline
 20& C  & N  & N  & N  & N  & N  & N  & N  & N  & N  & N  & N  & N  & N  & N  & N  & N & N \cr
 \hline
 \end{array}
 \]
\caption{3-Colorable Grids ($C$) and non 3-Colorable Grids ($N$)}\label{ta:col3}
\end{table}

\end{proof}

\section{Which Grids Can be 4-Colored?}\label{se:col4}

\subsection{Results that Use our Tools}

\begin{theorem}\label{th:col4}~
\begin{enumerate}
\item
$G_{41,5}$ and $G_{5,41}$ are not 4-colorable.
\item
$G_{31,6 }$  and $G_{6,31}$  are not 4-colorable.
\item
$G_{29,7 }$  and $G_{7,29}$ are not 4-colorable.
\item
$G_{25,9 }$  and $G_{9,25}$ are not 4-colorable.
\item
$G_{23,10}$  and $G_{10,23 }$ are not 4-colorable.
\item
$G_{22,11}$  and $G_{11,22 }$ are not 4-colorable.
\item
$G_{21,13}$  and $G_{13,21}$ are not 4-colorable.
\item
$G_{20,17}$  and $G_{17,20}$ are not 4-colorable.
\item
$G_{19,18}$  and $G_{18,19}$ are not 4-colorable.
\item
$G_{41,4}$ and $G_{4,41}$ are 4-colorable (this is trivial).
\item
$G_{40,5}$ and $G_{5,40}$ are 4-colorable.
\item
$G_{30,6}$ and $G_{6,30}$ are 4-colorable.
\item
$G_{28,8}$ and $G_{8,28}$ are 4-colorable.
\item
$G_{20,16}$ and $G_{16,20}$ are 4-colorable.
\end{enumerate}
\end{theorem}

\begin{proof}

We only consider grids $G_{n.m}$ where $n\ge m$.

\smallskip

\noindent
1,2,3,4,5,6,7,8,9)

In Table~\ref{ta:uncol4}  we show that several grids are not 4-colorable.
For each $(n,m)$ we use either Corollary~\ref{uncolor2} or \ref{uncolor3}.
In the table we give, for each $(n,m)$, the value of $\ceil{\frac{nm}{4}}$,
the $q,r$ such that $\ceil{\frac{nm}{4}}=qn+r$ with $0\le r\le n-1$,
which corollary we use (Use),  the premise of the corollary (Prem), and the arithmetic showing the premise is true 
(Arith).

\begin{table}[htbp]
\centering
\begin{tabular}{ c c c c c c c c }
$m$ & $n$ & $\ceil{\frac{nm}{4}}$ & $q$ & $r$ & Use & Prem& Arith \\
\hline\hline
5 & 41 & 52 & 1 & 11 & Cor~\ref{uncolor2} & $\frac{m(m-1)}{2} < r\le n$ &$10< 11 \le 41$\\
6 & 31 & 47 & 1 & 16 & Cor~\ref{uncolor2} & $\frac{m(m-1)}{2} < r\le n$ &$15< 16 \le 31$\\
7 & 29 & 51 & 1 & 22 & Cor~\ref{uncolor2} & $\frac{m(m-1)}{2} < r\le n$ &$21< 22 \le 29$\\
9 & 25 & 57 & 2 & 7  & Cor~\ref{uncolor3} & $\frac{m(m-1)-2qr}{q(q-1)} <n$ & $22 < 25$\\
10 & 23 & 58 & 2 & 12  & Cor~\ref{uncolor3} & $\frac{m(m-1)-2qr}{q(q-1)}<n$ & $21 < 23$\\
11 & 22 & 61 & 2 & 17  & Cor~\ref{uncolor3} & $\frac{m(m-1)-2qr}{q(q-1)} < n$ & $21 < 22$\\
13 & 21 & 69 & 3 & 6  & Cor~\ref{uncolor3} & $\frac{m(m-1)-2qr}{q(q-1)} < n$& $20 < 21$\\
17 & 20 & 85 & 4 & 5  & Cor~\ref{uncolor3} & $\frac{m(m-1)-2qr}{q(q-1)} < n$& $19 \frac{1}{3} < 20$\\
18 & 19 & 86 & 4 & 10  & Cor~\ref{uncolor3} & $\frac{m(m-1)-2qr}{q(q-1)} < n$ & $18 \frac{5}{6} < 19$\\
\hline
\end{tabular}
\caption{$(m,n)$ such that $G_{m,n}$ is not 4-colorable \label{ta:uncol4}}
\end{table}

\smallskip

\noindent
10) $G_{41,4}$ is clearly 4-colorable.

\smallskip

\noindent
11) $G_{40,5}$ is 4-colorable by Theorem~\ref{th:cplusone} with $c=4$.

\smallskip

\noindent
12) $G_{30,6}$ is 4-colorable by Theorem~\ref{th:cplusgen} with $c=4$ and $c'=2$.

\noindent
13) $G_{28,8}$ is 4-colorable by Theorem~\ref{th:primepower} 
with $p=2$, $d=3$, and $s=1$.

\smallskip

\noindent
14) $G_{20,16}$ is 4-colorable by Theorem~\ref{th:primepower} with
$p=2$, $d=2$, and $s=2$

\end{proof}

\begin{theorem}\label{th:19x17}
If $A\subseteq G_{19,17}$ and
$A$ is rectangle-free then $|A|\le 80=\ceil{\frac{19\cdot17}{4}}-1$.
Hence $G_{19,17}$ is not 4-colorable.
\end{theorem}

\begin{proof}
We divide the proof into cases.
Every case will either conclude that
$|A|\le 80$ or $A$ cannot exist.

For $1\le j\le 17$ let $x_j$ be the number of elements of
$A$ in column $j$.
We assume 

$$x_1\ge \cdots \ge x_{17}.$$

\begin{enumerate}
\item
$6\le x_1\le 19$. 

By Lemma~\ref{le:maxrf1} with $x=6$, $n=19$, $m=17$,

$$|A| \le x + m-1 +  \maxrf {n-x}{m-1} \le 6 + 17 - 1 + \maxrf{19-6}{17-1} = 22+\maxrf{13}{16}.$$

Assume, by way of contradiction, that $|A|\ge 81$.
Then $\maxrf{13}{16}\ge 59$
By Theorem~\ref{th:density} with $n=16$, $m=13$, $a=59$, $q=3$, $r=11$

$$16\le \floor{\frac{13\times 12-2\times 3\times 11}{3\times 2}}=15.$$

This is a contradiction.

\item
There exists $k$, $0\le k\le 12$, such that
$x_1=\cdots = x_k=5$ and $x_{k+1}\le 4$.
Then

$$|A| = \sum_{j=1}^{17} x_j = (\sum_{j=1}^{k} x_j) + (\sum_{j=k+1}^{17} x_j)
\le 5k + 4(17-k)= 68+k.$$

Since $k\le 12$ this quantity is $\le 68 + 12 = 80.$
Hence $|A|\le 80$.

\item
$x_1=x_2=\cdots=x_{13}=5$.
Look at the grid restricted to the first 13 columns.
Let $B$ be $A$ restricted to that grid.
Note that $B$ is a rectangle-free subset
of $G_{19,13}$ of size 65.
By Theorem~\ref{th:density} with
$n=19$, $m=13$, $a=65$, $q=3$, and $r=8$
we have

$$19\le \floor{\frac{13\times 12-2\times 8\times 3}{3\times 2}}=18.$$

This is a contradiction, hence $A$ cannot exist.
\end{enumerate}
\end{proof}

\begin{theorem}\label{th:24x9}
$G_{24,9}$ is 4-colorable.
\end{theorem}

\begin{proof}
Table~\ref{ta:s96} shows a 
strong $(4,1)$-coloring of $G_{9,6}$.
Apply Theorem~\ref{th:strong}
with $c=4$ and $c'=1$ to obtain a 4-coloring of $G_{24,9}$.

\begin{table}[htbp]
 \[
 \begin{array}{|c|c|c|c|c|c|c|c|}
\hline
1  & 2 & 2 & 1 & 2 & 2\cr
\hline
1  & 3 & 3 & 2 & 1 & 3\cr
\hline
1  & 4 & 4 & 3 & 3 & 1\cr
\hline
2  & 1 & 4 & 1 & 4 & 3\cr
\hline
3  & 1 & 2 & 3 & 1 & 4\cr
\hline
4  & 1 & 3 & 4 & 2 & 1\cr
\hline
4  & 3 & 1 & 1 & 3 & 4\cr
\hline
2  & 4 & 1 & 4 & 1 & 2\cr
\hline
3  & 2 & 1 & 2 & 4 & 1\cr
\hline
\end{array}
\]
\caption{Strong 4-coloring of $4_{9,6}$.\label{ta:s96}}
\end{table}
\end{proof}

It is an open question to generalize the construction in Theorem~\ref{th:24x9}.

\subsection{Results that Needed a Computer Program}

At this point in the paper the only grids whose 4-colorability is unknown
are $G_{22,10}$, $G_{21,11}$, $G_{21,12}$, $G_{17,17}$, $G_{17,18}$, and $G_{18,18}$.
This may seem like a computational problem that one could solve with a computer; however,
the number of possible 4-coloring of (say) $G_{18,18}$ is on the order of $4^{324}$.
By contrast, the number of protons in the universe, also called Eddington's number,
has been estimated at approximately $4^{128}$~\cite{protons}

The only technique we know of to show that $G_{n,m}$ is {\it not} $c$-colorable
is to show that no rectangle free set of $G_{n,m}$ is of size $\ge \ceil{\frac{nm}{c}}$.
In 2008 we obtained a rectangle free set of $G_{17,17}$ of size $74 = \ceil{\frac{17\times 17}{4}}+1$
(using a computer program).
Hence we were confident that $G_{17,17}$ is 4-colorable.
Using this rectangle free set as a starting point the number of possible 4-colorings
would be $4^{289-74}=4^{215}$ which is still larger than Eddington's number.
For all of the other grids $G_{n,m}$ that we did not know if they were 4-colorable,
a rectangle free set of size $\ceil{nm/4}$ was found. Hence either they are  all
4-colorable or there is a different technique to show grids are not $c$-colorable.

On November 30, 2009 William Gasarch (the second author on this paper) posted on his 
blog~\cite{17posedpost} {\it The $17\times 17$ challenge:} if someone emails 
William a 4-coloring of $G_{17,17}$ then he  will give them  \$289.00. 
An earlier version of this paper was posted as well.
Then the following happened:

\begin{enumerate}
\item
Brian Hayes, a popular science writer, put the problem on his blog~\cite{17hayes} thus
exposing the problem to many more people.
\item
Brad Larsen noticed that we didn't have a 4-coloring of $G_{21,11}$ and
$G_{22,10}$. He then found such 4-colorings using a SAT solver which, in his words,
took about 45 seconds.
\item
Many people worked on finding a 4-coloring of $G_{17,17}$ (for the money! for the glory!)
but could not solve it. This lead to speculation that
the problem may be difficult. Evidence for this was later found~\cite{gridnpc}, though by
that time a 4-coloring of $G_{17,17}$ had already been found. Irony?
\item
Bernd Steinbach and Christian Posthoff worked on solving the problem with  SAT solvers.
In a sequence of three brilliant papers they solved the problem
\cite{17solvedpaperprea,17solvedpaperpreb,17solvedpaper}.
This was very serious and deep research that may lead to improved SAT Solvers
for other problems.
They announced their result in February of 2012.
See \cite{17solvedpost} for the blog post about it.
Dr. Gasarch happily paid them the \$289.00.
\item
Marzio De Biasi easily found an extension of the 4-coloring of $G_{17,17}$ to $G_{18,18}$ and
posted it as a comment on~\cite{17solvedpost}.
Bernd Steinbach and Christian Posthoff had already known this coloring as well.
\item
Bernd Steinbach and Christian Posthoff used their techniques to find a 4-coloring of $G_{21,12}$
and posted it as a comment on the blog post~\cite{17solvedpost}.
With this $OBS_4$ was completely known!
\item
Inspired by the $17\times 17$ challenge and the solution to it
Neil Brewer and Dmitry Kamenetsky devised a contest 
at \url{http:infinitesearchspace.dyndns.org} 
that asked for the following: {\it For $c=1$ to 21 find the largest $n$
such that the $n\times n$ grid is $c$-colorable. You must also
present the coloring.} This lead to a lot of interesting discussion
including the following two points, one of which we use in our paper
\begin{enumerate}
\item
Tom Sirgedas obtained another 4-coloring of $G_{21,12}$.
To paraphrase him:
{\it
I noticed that the rectangle-free subset $A$ of $G_{21,12}$ in (the earlier version of) the paper had the following property:
If you viewed it as a $7\times 3$ grid of $3\times 3$ grids then in each of those $3\times 3$ grids
either all elements of the diagonal all in $A$ or none were in $A$.
I assumed that the solution would have this property. This cut down the number of possibilities by quite a lot.
Then, I just wrote an exhaustive depth-first-search to fill the grid one
color at a time, and each color one row at a time. Â I used a lot of pruning
and bitmasks, and solutions were found in a few minutes.
Unfortunately this approach seems to only work for this particular grid.
It won't scale well at all.
}
\item
Quimey Vivas posted a proof that if $c$ is prime then
$G_{c^2,c^2}$ is $c$-colorable. 
Ken Berg had previously send me a proof that $G_{c^2,c^2+c}$ is $c$-colorable
when $c$ is a power of a prime.
That proof is in this paper as Theorem~\ref{th:primepowersq}.
\end{enumerate}
\end{enumerate}

\begin{theorem}\label{th:21x12}
$G_{21,12}$ is 4-colorable
\end{theorem}

\begin{proof}
Bernd Steinbach and Christian Posthoff (as a team) and
Tom Sirgedas obtained a 4-coloring of $G_{21,12}$.
Tom Sirgedas's coloring is in Table~\ref{ta:2112}.

\begin{table}[htbp]
\[
\begin{array}{|c|c|c||c|c|c||c|c|c||c|c|c|}
\hline
1  & 2  & 2  & 3 &  2  & 1  & 3  & 4  & 4  & 3  & 1  & 3\cr
\hline
2 & 1 & 2 & 1 & 3 & 2 & 4 & 3 & 4 & 3 & 3 & 1\cr
\hline
2 & 2 & 1 & 2 & 1 & 3 & 4 & 4 & 3 & 1 & 3 & 3\cr
\hline
\hline
3 & 4 & 2 & 4 & 1 & 1 & 1 & 2 & 4 & 3 & 2 & 4\cr
\hline
2 & 3 & 4 & 1 & 4 & 1 & 4 & 1 & 2 & 4 & 3 & 2\cr
\hline
4 & 2 & 3 & 1 & 1 & 4 & 2 & 4 & 1 & 2 & 4 & 3\cr
\hline
\hline
2 & 4 & 1 & 3 & 2 & 4 & 2 & 3 & 3 & 4 & 2 & 1\cr
\hline
1 & 2 & 4 & 4 & 3 & 2 & 3 & 2 & 3 & 1 & 4 & 2\cr
\hline
4 & 1 & 2 & 2 & 4 & 3 & 3 & 3 & 2 & 2 & 1 & 4 \cr
\hline
\hline
 3 & 1 & 1 & 2 & 2 & 1 & 3 & 4 & 1 & 4 & 3 & 4 \cr
\hline
 1 & 3 & 1 & 1 & 2 & 2 & 1 & 3 & 4 & 4 & 4 & 3 \cr
\hline
 1 & 1 & 3 & 2 & 1 & 2 & 4 & 1 & 3 & 3 & 4 & 4 \cr
\hline
\hline
 3 & 4 & 4 & 1 & 3 & 2 & 2 & 4 & 2 & 1 & 1 & 3 \cr
\hline
 4 & 3 & 4 & 2 & 1 & 3 & 2 & 2 & 4 & 3 & 1 & 1 \cr
\hline
 4 & 4 & 3 & 3 & 2 & 1 & 4 & 2 & 2 & 1 & 3 & 1 \cr
\hline
\hline
 3 & 4 & 2 & 3 & 4 & 3 & 2 & 1 & 1 & 1 & 4 & 2 \cr
\hline
 2 & 3 & 4 & 3 & 3 & 4 & 1 & 2 & 1 & 2 & 1 & 4 \cr
\hline
 4 & 2 & 3 & 4 & 3 & 3 & 1 & 1 & 2 & 4 & 2 & 1 \cr
\hline
\hline
 3 & 3 & 1 & 4 & 2 & 4 & 4 & 1 & 3 & 2 & 1 & 2 \cr
\hline
 1 & 3 & 3 & 4 & 4 & 2 & 3 & 4 & 1 & 2 & 2 & 1 \cr
\hline
 3 & 1 & 3 & 2 & 4 & 4 & 1 & 3 & 4 & 1 & 2 & 2 \cr
\hline
\end{array}
\]
\caption{A 4-coloring of $G_{21,12}$ due to Tom Sirgedas.\label{ta:2112}}
\end{table}
\end{proof}

\begin{theorem}\label{th:22x10}
$G_{22,10}$ is 4-colorable
\end{theorem}

\begin{proof}
Brad Larsen obtained a 4-coloring of $G_{22,10}$:
We present it in Table~\ref{ta:2210}.

\begin{table}[htbp]
\[
\begin{array}{|c|c|c|c|c|c|c|c|c|c|}
\hline
 1  & 2  & 3  & 3  & 2  & 2  & 1  & 1  & 4  & 4\cr
\hline
 2  & 1  & 4  & 2  & 4  & 4  & 1  & 3  & 1  & 3\cr
\hline
 4  & 2  & 4  & 3  & 1  & 1  & 2  & 3  & 1  & 4\cr
\hline
 1  & 1  & 2  & 2  & 3  & 3  & 4  & 4  & 2  & 1\cr
\hline
 1  & 4  & 1  & 1  & 2  & 3  & 3  & 2  & 2  & 3\cr
\hline
 1  & 4  & 3  & 4  & 3  & 1  & 2  & 3  & 2  & 2\cr
\hline
 2  & 1  & 2  & 1  & 4  & 1  & 3  & 4  & 3  & 2\cr
\hline
 1  & 3  & 3  & 4  & 1  & 4  & 2  & 2  & 4  & 3\cr
\hline
 1  & 4  & 4  & 3  & 3  & 2  & 3  & 2  & 1  & 2\cr
\hline
 3  & 3  & 4  & 4  & 1  & 2  & 3  & 4  & 2  & 1\cr
\hline
 3  & 2  & 2  & 1  & 3  & 4  & 4  & 2  & 1  & 3\cr
\hline
 3  & 4  & 3  & 2  & 2  & 1  & 1  & 4  & 4  & 2\cr
\hline
 4  & 3  & 2  & 4  & 2  & 3  & 4  & 3  & 1  & 1\cr
\hline
 2  & 2  & 1  & 4  & 4  & 1  & 3  & 3  & 2  & 4\cr
\hline
 3  & 2  & 1  & 3  & 4  & 3  & 4  & 1  & 1  & 2\cr
\hline
 4  & 4  & 1  & 2  & 1  & 4  & 1  & 2  & 3  & 3\cr
\hline
 2  & 1  & 4  & 3  & 1  & 2  & 4  & 1  & 4  & 3\cr
\hline
 3  & 4  & 2  & 1  & 4  & 2  & 1  & 3  & 3  & 1\cr
\hline
 2  & 4  & 3  & 1  & 1  & 3  & 4  & 2  & 3  & 4\cr
\hline
 4  & 3  & 1  & 2  & 3  & 2  & 2  & 4  & 3  & 1\cr
\hline
 4  & 3  & 2  & 3  & 4  & 1  & 2  & 1  & 4  & 1\cr
\hline
 2  & 1  & 1  & 3  & 2  & 4  & 2  & 4  & 3  & 4\cr
\hline
\end{array}
\]
\caption{A 4-coloring of $2_{22,10}$ due to Brad Larsen\label{ta:2210}.}
\end{table}
\end{proof}

\begin{theorem}\label{th:18x18}
$G_{18,18}$ is 4-colorable
\end{theorem}

\begin{proof}
Bernd Steinbach and Christian Posthoff obtained 
a 4-coloring of $G_{18,18}$.
We present the coloring in Table~\ref{ta:1818}.

\begin{table}[htbp]
\[
\begin{array}{|c|c|c|c|c|c|c|c|c|c|c|c|c|c|c|c|c|c|}
\hline
 1 & 2 & 2 & 1 & 4 & 4 &  4 & 1 & 3 & 1 & 1 & 3 &  4 & 3 & 2 & 3 & 4 & 2 \cr
\hline
 3 & 1 & 4 & 4 & 4 & 3 &  3 & 2 & 4 & 1 & 2 & 2 &  1 & 1 & 2 & 3 & 2 & 3 \cr
\hline
 2 & 2 & 3 & 3 & 1 & 1 &  4 & 2 & 4 & 2 & 1 & 4 &  1 & 3 & 4 & 4 & 1 & 3 \cr
\hline
 1 & 1 & 3 & 2 & 4 & 3 &  1 & 2 & 1 & 4 & 4 & 3 &  2 & 4 & 3 & 4 & 1 & 2 \cr
\hline
 2 & 4 & 2 & 3 & 3 & 4 &  3 & 3 & 4 & 1 & 2 & 3 &  2 & 4 & 1 & 2 & 1 & 1 \cr
\hline
 3 & 4 & 4 & 1 & 1 & 2 &  2 & 2 & 1 & 4 & 3 & 3 &  3 & 1 & 4 & 2 & 4 & 1 \cr
\hline
 2 & 1 & 3 & 2 & 2 & 2 &  3 & 4 & 3 & 3 & 1 & 4 &  3 & 4 & 2 & 1 & 4 & 1 \cr
\hline
 4 & 1 & 4 & 3 & 1 & 2 &  4 & 1 & 2 & 2 & 2 & 1 &  3 & 4 & 3 & 3 & 3 & 2 \cr
\hline
 4 & 4 & 1 & 3 & 4 & 3 &  2 & 1 & 2 & 3 & 3 & 4 &  2 & 1 & 2 & 1 & 1 & 4 \cr
\hline
 2 & 3 & 3 & 4 & 3 & 4 &  2 & 1 & 1 & 4 & 3 & 4 &  1 & 2 & 1 & 3 & 2 & 2 \cr
\hline
 4 & 1 & 1 & 1 & 2 & 1 &  3 & 4 & 4 & 4 & 3 & 2 &  4 & 3 & 1 & 2 & 3 & 2 \cr
\hline
 3 & 2 & 3 & 4 & 2 & 1 &  2 & 3 & 1 & 1 & 2 & 1 &  4 & 4 & 4 & 1 & 3 & 4 \cr
\hline
 3 & 2 & 4 & 2 & 3 & 1 &  1 & 1 & 2 & 3 & 4 & 4 &  4 & 3 & 3 & 2 & 2 & 1\cr
\hline
 3 & 3 & 4 & 3 & 2 & 4 &  1 & 4 & 3 & 2 & 1 & 1 &  2 & 1 & 1 & 4 & 2 & 4\cr
\hline
 4 & 3 & 2 & 1 & 2 & 4 &  1 & 2 & 2 & 3 & 4 & 3 &  1 & 2 & 4 & 1 & 3 & 3\cr
\hline
 1 & 3 & 2 & 2 & 1 & 3 &  2 & 3 & 4 & 2 & 4 & 2 &  3 & 3 & 1 & 1 & 4 & 4\cr
\hline
 1 & 4 & 1 & 4 & 3 & 3 &  4 & 4 & 3 & 2 & 4 & 1 &  1 & 2 & 2 & 2 & 3 & 1\cr
\hline
 4 & 2 & 1 & 4 & 1 & 2 &  1 & 3 & 3 & 1 & 3 & 2 &  2 & 2 & 3 & 4 & 4 & 3\cr
\hline
\end{array}
\]
\caption{A 4-coloring of $2_{18,18}$ due to Bernd Steinbach and Christian Posthoff~\label{ta:1818}.}
\end{table}
\end{proof}

\begin{theorem}\label{th:obs4}
$$\OBS_4= 
\{ G_{41,5}, 
   G_{31,6}, 
   G_{29,7}, 
   G_{25,9}, 
   G_{23,10}, 
   G_{22,11}, 
   G_{21,13}, 
   G_{19,17} \} \bigcup $$

$$\{ 
   G_{17,19}, 
   G_{13,21}, 
   G_{11,22}, 
   G_{10,23}, 
   G_{9,25}, 
   G_{7,29}, 
   G_{6,31}, 
   G_{5,41}\}$$
\end{theorem}

\begin{proof}

We only deal with $G_{n,m}$ where $n\ge m$.
We show that, for all the grids $G_{n,m}$ listed where $a\ge b$,
$G_{n,m}$ is not 4-colorable but $G_{a-1,b}$ and $G_{n,m-1}$ are 4-colorable.

\begin{enumerate}
\item
$G_{41,5}$ is not 4-colorable by Theorem~\ref{th:col4}.
$G_{40,5}$ is 4-colorable by Theorem~\ref{th:col4}.
$G_{41,4}$ is clearly 4-colorable.
\item
$G_{31,6}$ is not 4-colorable by Theorem~\ref{th:col4}.
$G_{30,6}$  and $G_{31,5}$ are 4-colorable by Theorem~\ref{th:col4}.
\item
$G_{29,7}$ is not 4-colorable by Theorem~\ref{th:col4}.
$G_{28,7}$ and $G_{29,6}$  are 4-colorable by Theorem~\ref{th:col4}.
\item
$G_{25,9}$ is not 4-colorable by Theorem~\ref{th:col4}.
$G_{24,9}$ is 4-colorable by Theorem~\ref{th:24x9}.
$G_{25,8}$ is 4-colorable by Theorem~\ref{th:col4}.
\item
$G_{23,10}$ is not 4-colorable by Theorem~\ref{th:col4}.
$G_{22,10}$ is 4-colorable by Theorem~\ref{th:22x10}.
$G_{23,9}$ is 4-colorable by Theorem~\ref{th:24x9}.
\item
$G_{22,11}$ is not 4-colorable by Theorem~\ref{th:col4}.
$G_{22,10}$ is 4-colorable by Theorem~\ref{th:22x10}.
$G_{21,11}$ is 4-colorable by Theorem~\ref{th:21x12}.
\item
$G_{21,13}$ is not 4-colorable by Theorem~\ref{th:col4}.
$G_{20,13}$ is 4-colorable by Theorem~\ref{th:col4}.
$G_{21,12}$ is 4-colorable by Theorem~\ref{th:21x12}.
\item
$G_{19,17}$ is not 4-colorable by Theorem~\ref{th:19x17}.
$G_{18,17}$ is 4-colorable by Theorem~\ref{th:18x18}.
$G_{19,16}$ is 4-colorable by Theorem~\ref{th:col4}.
\end{enumerate}

The following chart indicates exactly which grids are 4-colorable.
The entry for $(n,m)$ is $C$ if $\Gnm$ is 4-colorable, and
$N$ if $\Gnm$ is not 4-colorable.
From the chart one easily sees that the grids listed in this theorem
are the only elements of $\OBS_4$.

\begin{table}[htbp]
 \[
 \begin{array}{|c|c|c|c|c|c|c|c|c|c|c|c|c|c|c|c|c|c|c|}
 \hline
   & 04 & 05 & 06 & 07 & 08 & 09 & 10 & 11 & 12 & 13 & 14 & 15 & 16 & 17 & 18 & 19 & 20 & 21 \cr
 \hline
  8& C  & C  & C  & C  & C  & C  & C  & C  & C  & C  & C  & C  & C  & C  & C  & C  & C & C \cr
 \hline
  9& C  & C  & C  & C  & C  & C  & C  & C  & C  & C  & C  & C  & C  & C  & C  & C  & C & C \cr
 \hline
 10& C  & C  & C  & C  & C  & C  & C  & C  & C  & C  & C  & C  & C  & C  & C  & C  & C & C \cr
 \hline
 11& C  & C  & C  & C  & C  & C  & C  & C  & C  & C  & C  & C  & C  & C  & C  & C  & C & C \cr
 \hline
 12& C  & C  & C  & C  & C  & C  & C  & C  & C  & C  & C  & C  & C  & C  & C  & C  & C & C \cr
 \hline
 13& C  & C  & C  & C  & C  & C  & C  & C  & C  & C  & C  & C  & C  & C  & C  & C  & C & N \cr
 \hline
 14& C  & C  & C  & C  & C  & C  & C  & C  & C  & C  & C  & C  & C  & C  & C  & C  & C & N \cr
 \hline
 15& C  & C  & C  & C  & C  & C  & C  & C  & C  & C  & C  & C  & C  & C  & C  & C  & C & N \cr
 \hline
 16& C  & C  & C  & C  & C  & C  & C  & C  & C  & C  & C  & C  & C  & C  & C  & C  & C & N \cr
 \hline
 17& C  & C  & C  & C  & C  & C  & C  & C  & C  & C  & C  & C  & C  & C & C & N  & N  & N\cr
 \hline
 18& C  & C  & C  & C  & C  & C  & C  & C  & C  & C  & C  & C  & C  & C & C & N  & N  & N\cr
 \hline
 19& C  & C  & C  & C  & C  & C  & C  & C  & C  & C  & C  & C  & C  & N  & N  & N  & N & N\cr
 \hline
 20& C  & C  & C  & C  & C  & C  & C  & C  & C  & C  & C  & C  & C  & N  & N  & N  & N & N\cr
 \hline
 21& C  & C  & C  & C  & C  & C  & C  & C  & C & N  & N  & N  & N  & N  & N  & N  & N & N\cr
 \hline
 22& C  & C  & C  & C  & C  & C  & C & N  & N  & N  & N  & N  & N  & N  & N  & N  & N & N\cr
 \hline
 23& C  & C  & C  & C  & C  & C  & N  & N  & N  & N  & N  & N  & N  & N  & N  & N  & N & N\cr
 \hline
 24& C  & C  & C  & C  & C  & C  & N  & N  & N  & N  & N  & N  & N  & N  & N  & N  & N & N\cr
 \hline
 25& C  & C  & C  & C  & C  & N  & N  & N  & N  & N  & N  & N  & N  & N  & N  & N  & N & N\cr
 \hline
 26& C  & C  & C  & C  & C  & N  & N  & N  & N  & N  & N  & N  & N  & N  & N  & N  & N & N \cr
 \hline
 27& C  & C  & C  & C  & C  & N  & N  & N  & N  & N  & N  & N  & N  & N  & N  & N  & N & N\cr
 \hline
 28& C  & C  & C  & C  & C  & N  & N  & N  & N  & N  & N  & N  & N  & N  & N  & N  & N & N\cr
 \hline
 29& C  & C  & C  & N  & N  & N  & N  & N  & N  & N  & N  & N  & N  & N  & N  & N  & N & N \cr
 \hline
 30& C  & C  & C  & N  & N  & N  & N  & N  & N  & N  & N  & N  & N  & N  & N  & N  & N & N\cr
 \hline
 31& C  & C  & N  & N  & N  & N  & N  & N  & N  & N  & N  & N  & N  & N  & N  & N  & N & N\cr
 \hline
 32& C  & C  & N  & N  & N  & N  & N  & N  & N  & N  & N  & N  & N  & N  & N  & N  & N  & N \cr
 \hline
 33& C  & C  & N  & N  & N  & N  & N  & N  & N  & N  & N  & N  & N  & N  & N  & N  & N & N\cr
 \hline
 34& C  & C  & N  & N  & N  & N  & N  & N  & N  & N  & N  & N  & N  & N  & N  & N  & N & N\cr
 \hline
 35& C  & C  & N  & N  & N  & N  & N  & N  & N  & N  & N  & N  & N  & N  & N  & N  & N & N\cr
 \hline
 36& C  & C  & N  & N  & N  & N  & N  & N  & N  & N  & N  & N  & N  & N  & N  & N  & N & N\cr
 \hline
 37& C  & C  & N  & N  & N  & N  & N  & N  & N  & N  & N  & N  & N  & N  & N  & N  & N  & N\cr
 \hline
 38& C  & C  & N  & N  & N  & N  & N  & N  & N  & N  & N  & N  & N  & N  & N  & N  & N  & N\cr
 \hline
 39& C  & C  & N  & N  & N  & N  & N  & N  & N  & N  & N  & N  & N  & N  & N  & N  & N & N\cr
 \hline
 40& C  & C  & N  & N  & N  & N  & N  & N  & N  & N  & N  & N  & N  & N  & N  & N  & N & N\cr
 \hline
 41& C  & N  & N  & N  & N  & N  & N  & N  & N  & N  & N  & N  & N  & N  & N  & N  & N & N\cr
 \hline
 \end{array}
 \]
\caption{4-Colorable Grids ($C$) and non 4-Colorable Grids ($N$)}\label{ta:col4}
\end{table}
\end{proof}

\section{Application to Bipartite Ramsey Numbers}\label{se:bipartite}

We state the Bipartite Ramsey Theorem.
See~\cite{GRS} for history, details, and proof.

\begin{definition}
A complete bipartite graph, $G = (V_1,V_2, E)$, is a bipartite graph such that for any two vertices, $v_1 \in V_1$ and $v_2 \in V_2$, $(v_1, v_2)$ is an edge in $G$. The complete bipartite graph with partitions of size $|V_1| = a$ and $|V_2| = b$, is denoted $K_{a,b}$.

\end{definition}

\begin{theorem}
For all $a,c$ there exists $n=BR(a,c)$ such that for all
$c$-colorings of the edges of $K_{n,n}$ there
will be a monochromatic $K_{a,a}$.
\end{theorem}

The following theorem is easily seen to be
equivalent to this.

\begin{theorem}
For all $a,c$ there exists $n=BR(a,c)$ so that for all
$c$-colorings of $G_{n,n}$ there will be a monochromatic
$a\times a$ submatrix.
\end{theorem}

In this paper we are $c$-coloring $G_{n,m}$ and looking
for a $2\times 2$ monochromatic submatrix.
We have the following theorems which, except where noted,
seem to be new.

\begin{theorem}\label{th:bipartite}~
\begin{enumerate}
\item
$BR(2,2) = 5$.
(This was also shown in~\cite{bipartitenum}.)
\item
$BR(2,3) = 11$.
\item
$ BR(2,4)=19$.
\item
$BR(2,c) \le c^2+c$.
\item
If $p$ is a prime and $s\in\natt$, then $BR(2,p^s)> p^{2s}$.
\item
For almost all $c$, $BR(2,c) \ge c^2 - 2c^{1.525}+c^{1.05}.$
\end{enumerate}
\end{theorem}

\begin{proof}

\noindent
1) By Theorem~\ref{th:col2},
$G_{5,5}$ is not 2-colorable
and $G_{4,4}$ is 2-colorable.

\smallskip

\noindent
2) By Theorem~\ref{th:11x10},
$G_{11,11}$ is not 3-colorable.
By Theorem~\ref{th:10x10}
$G_{10,10}$ is 3-colorable.

\smallskip

\noindent
3) By Theorem~\ref{th:19x17},
$G_{19,19}$ is not 4-colorable.
By Theorem~\ref{th:18x18}
$G_{18,18}$ is 4-colorable.

\smallskip

\noindent
4) By Theorem~\ref{th:csq},
$G_{c^2+c,c^2+c}$ is not $c$-colorable.

\smallskip

\noindent
5) By Theorem~\ref{th:primepower}, 
$G_{cr,cm}$ is $c$-colorable where
$c=p^s$, $r=p^s$, and $m=\frac{p^{2s}-1}{p^s-1}$.
Note that $m\ge p^s$. Hence
$G_{p^{2s},p^{2s}}$ is $p^s$-colorable.

\smallskip

\noindent
6) Baker, Harman,
and Pintz~\cite{BHP} (see~\cite{ConPrimesSurvey} for
a survey)
showed that
for almost all $c$, there is a prime between
$c$ and $c-c^{0.525}$.
Let $p$ be that prime. By part 5 with $s=1$,
$BR(2,p)\ge p^2$.
Hence

$$
BR(2,c) \ge BR(2,p) \ge p^2 \ge (c-c^{0.525})^2 =
c^2 - 2c^{1.525} + c^{1.05}.
$$
\end{proof}

\section{Open Questions}

\begin{enumerate}
\item
Refine our tools so that our ugly proofs can be corollaries
of our tools.
\item
Find an algorithm that will, given
$c$, find $\OBS_c$ or $|\OBS_c|$ quickly.
\item
We know that $2\sqrt{c}(1-o(1)) \le |\OBS_c| \le 2c^2$.
Bring these bounds closer together.
\item
All of our results of the form {\it $G_{n,m}$ is not $c$-colorable}
have the same type of proof: show that there is no rectangle free subset of
$G_{n,m}$ of size $\ceil{ab/c}$. Either
\begin{itemize}
\item
show that if a grid $G_{n,m}$ has a rectangle free set of size $\ceil{nm/c}$
then it is $c$-colorable, or
\item
develop some other technique to show grids are not $c$-colorable.
\end{itemize}
\item
Find $OBS_5$ and beyond!
\end{enumerate}

\section{Acknowledgments}

We would like to thank the following people for providing us with colorings:
\begin{enumerate}
\item
Brad Larsen for providing us with a 4-colorings of $G_{22,10}$,
\item
Bernd Steinbach and Christian Posthoff for providing us with 4-colorings of
$G_{21,12}$ and $G_{18,18}$.
\item
Tom Sirgedas for providing us with another 4-colorings of $G_{21,12}$,
\end{enumerate}

We thank Ken Berg and Quimey Vivas for providing us with
proofs that, for $c$ a prime power,  $G_{c^2,c^2}$ is $c$-colorable.
We would also like to thank Ken Berg for
the proof that, for $c$ a prime power,  $G_{c^2,c^2+c}$ is $c$-colorable.

We thank Michelle Burke, Brett Jefferson,
and Krystal Knight who worked with the second and third authors over the
Summer of 2006 on this problem.
As noted earlier, Brett Jefferson has his own
paper on this subject~\cite{gridsj}.

We thank
Nils Molina, Anand Oza, and Rohan Puttagunta
who worked with the second author in Fall 2008 on variants
of the problems presented here. They won the Yau prize
for their work.

We thank L{\'a}szl{\'o} Sz{\'e}kely
for pointing out the connection to bipartite Ramsey numbers,
Larry Washington for providing information on number theory
that was used in the proof of Theorem~\ref{th:bipartite},
Russell Moriarty for proofreading and intelligent commentary,

\section{Appendix: Exact Values of $\maxrf n m$ for $0\le n\le 6$, $m\le n$}

\begin{lemma}\label{le:values}~
\begin{enumerate}
\item[0)]
For $m\ge 0$, $\maxrf 0 m = 0$.
\item[1)]
For $m\ge 1$, $\maxrf 1 m = m$.
\item[2)]
For $m\ge 2$, $\maxrf 2 m = m+1$.
\item[3)]
For $m\ge 3$, $\maxrf 3 m  = m+3$.
\item[4)]
\begin{displaymath}
\maxrf 4 m =
\begin{cases}
m+5 \hbox{ if } 4\le m\le 5\\
m+6 \text{ if } m\ge 6 \\
\end{cases}
\end{displaymath}

\item[5)]
\begin{displaymath}
\maxrf 5 m =
\begin{cases}
12 \text{ if } m=5\\
m+8 \text{ if } 6\le m\le 7 \\
m+9 \text{ if } 8\le m\le 9 \\
m+10 \text{ if } m\ge 10 \\
\end{cases}
\end{displaymath}

\item[6)]
\begin{displaymath}
\maxrf 6 m =
\begin{cases}
2m+4 \text{ if } 6\le m \le 7 \\
19 \text{ if } m=8 \\
m+12\text{ if }  9\le m\le 10 \\
m+13 \text{ if }  11 \le m\le 12 \\
m+14 \text{ if }  13 \le m\le 14 \\
m+15 \text{ if }  m\ge 15\\
\end{cases}
\end{displaymath}

\end{enumerate}
\end{lemma}

\begin{proof}

Theorem~\ref{th:density} will provide all of
the upper bounds. The lower bounds are obtained
by actually exhibiting rectangle-free sets of
the appropriate size. We do this for the case of
$\maxrf 6 m $. Our technique applies to all of the other cases.

\smallskip

\noindent
{\bf Case 1: $\maxrf 6 m $ where $6\le m\le 7$ and $m=8$:}
Fill the first four columns with 3 elements (all pairs
overlapping). 
Each column of 3 blocks exactly $\binom{3}{2}=3$
of the possible $\binom{6}{3}=15$ ordered pairs,
hence 12 are blocked.
Hence we can fill the next $15-12=3$ columns with
two elements each, and the remaining column (if $m=8$)
with 1 element.
The picture below shows the result for $\maxrf 6 8 =19$;
however, if you just look at the first 6 (7) columns you get
the result for $\maxrf 6 6 $ ($\maxrf 6 7 $).

\[
\begin{array}{|c|c|c|c|c|c|c|c|}
\hline
R& &R& &R& & & R\cr
\hline
R& & &R& &R &  & \cr
\hline
R&R& & & &  &R &\cr
\hline
 &R&R& & &R & & \cr
\hline
 &R& &R&R&  & & \cr
\hline
 & &R&R& &  &R &\cr
\hline
\end{array}
\]

\smallskip

\noindent
{\bf Case 2: $\maxrf 6 m $ where $9\le m\le 10$:}
Fill the first three columns with 3 elements each
(all pairs overlapping). 
Each column of 3 blocks exactly $\binom{3}{2}=3$
of the possible $\binom{6}{3}=15$ ordered pairs,
hence 9 are blocked.
Hence we can fill the next $15-9=6$ columns with
two elements each and the remaining column (if $m=10$)
with 1 element.
The picture below shows the result for $\maxrf 6 {10} =22$;
however, if you just look at the first 9 columns you get
the result $\maxrf 6 9 =21$.

\[
\begin{array}{|c|c|c|c|c|c|c|c|c|c|}
\hline
R& &R& &R&  &  &   & & \cr
\hline
R& & & & &R &  & R &R& \cr
\hline
R&R& & & &  &R &   & & \cr
\hline
 &R&R& & &R &  &   & & \cr
\hline
 &R& &R&R&  &  & R & & \cr
\hline
 & &R&R& &  &R &   &R&R\cr
\hline
\end{array}
\]

\smallskip

\noindent
{\bf Case 3: $\maxrf 6 m $ where $11\le m\le 12$:}
Fill the first two columns with 3 elements each
(they overlap).
Each column of 3 blocks exactly $\binom{3}{2}=3$
of the possible $\binom{6}{3}=15$ ordered pairs,
hence 6 are blocked.
Hence we can fill the next $15-6=9$ columns with
two elements each and the remaining column (if $m=12$)
with 1 element.
The picture below shows the result for $\maxrf 6 {12} =25$;
however, if you just look at the first 11 columns you get
the result $\maxrf 6 {11} =24$.

\[
\begin{array}{|c|c|c|c|c|c|c|c|c|c|c|c|}
\hline
R& &R& &R&  &  &   & & &R& \cr
\hline
R& & & & &R &  & R &R& & & \cr
\hline
R&R& & & &  &R &   & & & & \cr
\hline
 &R&R& & &R &  &   & &R& & \cr
\hline
 &R& &R&R&  &  & R & & & & \cr
\hline
 & & &R& &  &R &   &R&R &R&R\cr
\hline
\end{array}
\]

\smallskip
\noindent
{\bf Case 4: $\maxrf 6 m $ where $13\le m\le 14$:}
Fill the first column with 3 elements.
This column of 3 blocks exactly $\binom{3}{2}=3$
of the possible $\binom{6}{3}=15$ ordered pairs.
Hence we can fill the next $15-3=12$ columns with
two elements each and the remaining column (if $m=14$)
with 1 element.
We omit the picture.

\smallskip
\noindent
{\bf Case 5: $\maxrf 6 m $ where $ m\ge 15$:}
Fill the first $\binom{6}{2}=15$ columns with two
elements each in a way so that each column has a distinct pair.
Fill the remaining $m-15$ columns with one element each.
The result is a rectangle-free set of size 
$30 + m-15 = m+15$.
\end{proof}

\end{document}